\documentclass[a4paper,11pt]{article} 

\usepackage[utf8]{inputenc}
\usepackage{float}

\usepackage{graphicx}
\usepackage{amsmath}
\usepackage{amsthm}
\usepackage{amssymb}

\usepackage[all,cmtip]{xy}
\usepackage[english]{babel}
\usepackage{enumitem}
\usepackage{tikz-cd}
\usepackage{mathtools}
\usepackage{pinlabel}
\usepackage{hyperref}
\usepackage{mathtools}
\usepackage{titling}

\usepackage{changes}
\colorlet{Changes@Color}{red}

\pretitle{\begin{center}\Huge\bfseries}
\posttitle{\par\end{center}\vskip 0.5em}
\preauthor{\begin{center}\Large\ttfamily}
\postauthor{\end{center}}
\predate{\par\large\centering}
\postdate{\par}

\hypersetup{
    colorlinks=true,
    linkcolor=blue,
    filecolor=magenta,      
    urlcolor=cyan,
}

\newcommand{\Z}{\mathbb Z}

\newcommand{\Q}{\mathbb Q}
\newcommand{\tsb}[1]{\textnormal{#1}}
\DeclareMathOperator{\Cyc}{Cyc}
\DeclareMathOperator{\Pic}{Pic}
\DeclareMathOperator{\Hom}{Hom}
\DeclareMathOperator{\Ext}{Ext}

\DeclareMathOperator{\DivRat}{DivRat}
\DeclareMathOperator{\DivPrinc}{DivPrinc}
\DeclareMathOperator{\Div}{Div}
\DeclareMathOperator{\codim}{codim}
\DeclareMathOperator{\rank}{rank}
\DeclareMathOperator{\Spec}{Spec}

\newcommand{\hop}{\vskip .3cm\noindent} 

\newtheorem{thm}{Theorem}[section] 
\newtheorem{cor}[thm]{Corollary}
\newtheorem{prop}[thm]{Proposition}
\newtheorem{lem}[thm]{Lemma}
\newtheorem{defi}[thm]{Definition}
\newtheorem{rema}[thm]{Remark}
\theoremstyle{definition}
\newtheorem{exam}[thm]{Example}

\makeatletter
\newcommand{\setword}[2]{%
  \phantomsection
  #1\def\@currentlabel{\unexpanded{#1}}\label{#2}%
}
\makeatother

\begin{document}

\title{Extending torsors over regular models of curves}

\author{Sara Mehidi}

\maketitle

\begin{abstract}
Let $R$ be a discrete valuation ring with field of fractions $K$ and residue field $k$ of characteristic $p>0$.
Given a commutative finite group scheme $G$ over $K$ and a smooth projective curve $C$ over $K$ with a rational point, we study the extension of pointed fppf $G$-torsors over $C$ to pointed torsors over some $R$-regular model $\mathcal{C}$ of $C$. We first study this problem in the category of log schemes: given a finite flat $R$-group scheme $\mathcal{G}$, we prove that the data of a pointed $\mathcal{G}$-log torsor over $\mathcal{C}$ is equivalent to that of a morphism $\mathcal{G}^D \to \Pic^{log}_{\mathcal{C}/R}$, where $\mathcal{G}^D$ is the Cartier dual of $\mathcal{G}$ and $\Pic^{log}_{\mathcal{C}/R}$ the log Picard functor. After that, we give a sufficient condition for such a log extension to exist, and then we compute the obstruction for the existence of an extension in the category of usual schemes. \\

In a second part, we generalize a result of Chiodo \cite{Chiodo} which gives a criterion for the $r$-torsion subgroup of the N\'eron model of $J$ to be a finite flat group scheme, and we combine it with the results of the first part. 
Finally, we give a detailed example of extension of torsors when $C$ is a hyperelliptic curve defined over $\Q$, which illustrates our techniques.
\end{abstract}
\tableofcontents

\section*{Introduction}

Throughout this paper, $R$ denotes a discrete valuation ring with field of fractions $K$ and residue field $k$ of characteristic $p >0$.  Let $f: \mathcal{X} \to \Spec(R)$ be a faithfully flat morphism of finite type, and let $f_K: X \to \Spec(K)$ be its generic fiber. Assume that we are given a finite $K$-group scheme $G$ and an fppf $G$-torsor $Y \to X$. The problem of extending the $G$-torsor $Y \to X$ consists in finding a finite and flat $R$-group scheme $\mathcal{G}$ whose generic fiber is isomorphic to $G$, and an fppf $\mathcal{G}$-torsor $\mathcal{Y} \to \mathcal{X}$ whose generic fiber is isomorphic to $Y \to X$ as a $G$-torsor. \\ 

A general solution to this problem does not exist, but this question has been investigated in various settings by Grothendieck, Tossici \cite{Tossici} and Antei \cite{Ant}, amongst others. For example, the case when $G$ is a constant group scheme of order coprime to $p$ and $\mathcal{X}\to \Spec(R)$ is smooth with geometrically connected fibers does have a solution; see \cite[X, \S 3.1 and \S 3.6]{Groth2}.\\

 A natural strategy consists in first looking for a model of $G$ that is finite and flat (if any), and then focusing on the extension of the $G$-torsor. For instance, this has been done by Tossici  in the case where $p$ divides $|G|$ \cite{Tossici}. He studied the extension of torsors under commutative finite flat group schemes over local schemes under some extra assumptions. Moreover he also studied, using the so-called effective models, the extension of $\Z/p\Z$-torsors and $\Z/p^2\Z$-torsors imposing the normality of $\mathcal{Y}$. \\ 
  
Antei and Emsalem approached the issue with a different point of view in \cite{Antei}. With the assumption that $G$ is affine (but not necessarily commutative), since an $R$-model of $G$ that is finite and flat does not always exist, they choose to work with a more general model of $G$ that is flat but only quasi-finite, and then extend the torsor over  some scheme $\mathcal{X}'$, which is obtained by modifying the special fiber of $\mathcal{X}$.  By allowing such models of $G$, they solved the problem of extending any $G$-torsor up to a modification of $\mathcal{X}$, without any assumptions on the residue characteristic. When $\mathcal{X}$ is a relative curve, this modification is obtained by performing a finite sequence of N\'eron blow-ups of $\mathcal{X}$ along closed subschemes of the special fiber. \\

Now, if $Q$ is a rational point on $X$, the data of a pointed torsor over $X$ (relative to $Q$) is equivalent to that of a torsor which is trivial when restricted to $Q$. Futhermore, we have the following bijective correspondence:
\begin{align}\label{tor-mor}
\begin{split}
 \{ \mathrm{pointed~fppf}~G\mathrm{-torsors~over}~X\mathrm{,~relative~to~}Q \} & \to \Hom(G^D, \Pic_{X/K})\\
 Y & \mapsto h_{Y}
 \end{split}
\end{align}
where $G^D$ is the Cartier dual of $G$ and $\Pic_{X/K}$ is the relative Picard functor of $X$ over $K$. In fact, this bijective correspondence also holds over $R$ and in an earlier work \cite{Ant}, Antei used it to study the problem of extending torsors. In particular, it says that the extension of torsors reduces to the extension of certain group schemes and morphisms between them. Given some quite strong assumption on the Picard functor  $\Pic_{\mathcal{X}/R}$,  Antei treated essentially the case where $\mathcal{X}$ is smooth. He proved that $G$-torsors always extend in this context \cite[Theorem 3.10]{Ant}. \\

In this paper, we shall consider the problem of extending fppf $G$-torsors over a smooth projective $K$-curve $C$, endowed with a $K$-rational point $Q$, and seek an extension over some $R$-regular model $\mathcal{C}$ of $C$.
We first emphasize that the existence of an extension of a given $G$-torsor is, in general, a strong requirement: if we assume that we have a finite flat model $\mathcal{G}$ of $G$, that is in addition \'etale, our extended torsor --if it exists-- should be unramified. But this is quite a strong condition; in order to relax this condition, we shall work inside a larger category, namely the category of \textit{logarithmic} torsors. More precisely, we endow $\mathcal{C}$  with the logarithmic structure induced by its special fiber $\mathcal{C}_k$, seen as a divisor. Then logarithmic torsors over $\mathcal{C}$ are, roughly speaking, tamely ramified along $\mathcal{C}_k$. \\

In view of \cite{Ant} discussed above, the first natural question  is to investigate the correspondence between torsors and homomorphisms of group schemes in the logarithmic setting. Then we prove the following:\\

\textbf{Theorem \ref{raynaudlog}.} \textsl{Let $C$ be a smooth projective and geometrically connected curve over $K$, endowed with a rational point $Q$. Let $f: \mathcal{C} \to \Spec(R)$ be a regular model of $C$. Let $\mathcal{G}$ be a commutative finite flat group scheme over $R$ and let $\mathcal{G}^D$ denote its Cartier dual. We have a canonical isomorphism:
 $$R^1_{klf}f_*(\mathcal{G}_{\mathcal{C}}^D) \xrightarrow{\simeq} \underline{\Hom}(\mathcal{G}, \Pic^{log}_{\mathcal{C}/R}) $$
 where $\Pic^{log}_{\mathcal{C}/R}$ is the logarithmic Picard functor (cf. Definition \ref{Piclog})}\hop \hop
 The global sections of the sheaf on the left can be interpreted as $\mathcal{G}^D_{\mathcal{C}}$-pointed log torsors over $\mathcal{C}$ (cf. Corollary \ref{pointedlog}).
 Hence, this shows that extending torsors into log torsors over $\mathcal{C}$ also reduces to the extension of some group functors and morphisms between them. However, this criterion is not easy to handle in practice; e.g. we do not have any general result about the representability of the log Picard functor. However, we can deduce from it the following sufficient condition for extension to log torsors:\\

\textbf{Corollary \ref{G^D -> J}.} \textsl{Let $G$ be a commutative finite $K$-group scheme and let $Y \to C$ be an fppf pointed $G$-torsor (relative to $Q$). Let $J$ be the Jacobian of the curve $C$. If the morphism $h_{Y}: G^D \to J$ extends into a morphism $\mathcal{G}^D \to \mathcal{J}$, where $\mathcal{J}$ is the Néron model of $J$ and $\mathcal{G}$ a finite flat $R$-group scheme with generic fiber $G$, then the $G$-torsor $Y \to C$ extends uniquely into a logarithmic pointed $\mathcal{G}$-torsor over $\mathcal{C}$.}\\

We prove that the closed immersion $J \hookrightarrow \Pic_{C/K}$ extends into a morphism $\mathcal{J} \to \Pic^{log}_{\mathcal{C}/R}$, and we call Néron-log torsors all the log torsors such that the associated morphism of Theorem \ref{raynaudlog} factors through $\mathcal{J}$. We give a description for these torsors in terms of the Poincaré log extension of $\mathcal{J}_{\mathcal{C}}$ by $\mathbb{G}_{m,\mathcal{C}}$. In addition, we compute the obstruction for a Néron-log torsor to be fppf, and hence the obstruction for the generic torsor to extend into an fppf one in this context. We will see that this obstruction can be expressed using the obstruction for the Poincaré extension of $J_C$ by $\mathbb{G}_{m,C}$ to extend into an fppf extension of $\mathcal{J}_{\mathcal{C}}$ by $\mathbb{G}_{m,\mathcal{C}}$. \\

In a second part,  we generalize a result by Chiodo \cite[Propositions 7.4.1 and 7.5.1]{Chiodo} which provides a criterion for the existence of a finite \'etale $R$-model of $J[r]$ when $r$ is prime to $p$. In fact, his result is a finiteness criterion for $\mathcal{J}[r]$ because when $r$ is prime to $p$, $\mathcal{J}[r]$ is \'etale, hence is the natural candidate for being the N\'eron model of $J[r]$. This no longer holds when $p$ divides $r$; nevertheless, we show that the same finiteness criterion for $\mathcal{J}[r]$ holds. This yields interesting examples of commutative group schemes admitting a finite flat model that maps into $\mathcal{J}$. Applying our previous results to this setting, we obtain the following:\\

 \noindent\textbf{Corollary \ref{coro}.}
\textsl{Let $C$ be a smooth projective  geometrically connected curve over $K$ of genus $g \geq 2$, with a rational point, and let $\mathcal{C}$ be a regular model of $C$ over $R$. Let $G$ be a commutative finite $K$-group scheme killed by $r$, and let $Y \rightarrow C$ be a pointed  fppf $G$-torsor such that $Y$ is connected. If $C$ is semistable and if Chiodo's criterion is satisfied (Theorem~\ref{prop:CC}), then $G$ has a finite flat $R$-model $\mathcal{G}$ for which $Y \rightarrow C$ admits a unique extension into a pointed logarithmic $\mathcal{G}$-torsor over $\mathcal{C}$.}\\

Finally, the last part of the paper is devoted to the study of an example of extension of torsors. We give ourselves a hyperelliptic curve over $\Q$ whose Jacobian contains a subgroup isomorphic to $(\Z/p\Z)^2$. Hence this gives a $\mu_{p}^2$ -torsor on the curve. In the case where $p=3$, we first construct a regular model of the curve above $\Z_3$. Then, we will ask if the previous torsor extends to this model, and this is achieved by studying the (unique) extension of the morphism $(\Z/3\Z)^2 \to J$ into a morphism $(\Z/3\Z)^2 \to \mathcal{J}$. \\

 This paper is divided into four sections. In section \ref{section1}, we recall some facts about log schemes and log torsors. Then, we prove Theorem \ref{raynaudlog} and Corollary \ref{G^D -> J} stated above. In section \ref{section2}, we compute the obstruction for an fppf torsor that extends into a log torsor (under the assumptions of Corollary \ref{G^D -> J}) to extend into an fppf torsor. Section \ref{section3} is devoted to the generalization of the result of Chiodo stated before, and the application of our main theorems to get Corollary \ref{coro}. Finally, in section \ref{section4} we study examples of extension of torsors over a given hyperelliptic curve.  \\
 
 Throughout this article, all schemes and log schemes are assumed to be locally noetherian.\\

 \textbf{Acknowledgements}.
The author would like to warmly thank her thesis advisors Jean Gillibert and Dajano Tossici for their support and encouragements, and the long hours of discussion they devoted to her, without which this article would not have been completed. She would also like to thank her colleague William Dallaporta with whom working on the construction of regular models was very enriching. Finally, the author would like to warmly thank the referee for their careful reading of the paper and their very interesting comments and suggestions which helped to improve the paper.

\section{Extension of torsors into log and fppf torsors}\label{section1}
\subsection{Logarithmic schemes and logarithmic torsors}

We start this section by recalling some definitions about logarithmic schemes and the way they are used through this paper (we often write \textit{log} instead of logarithmic for simplification). For a detailed introduction to log schemes, we refer the reader to \cite{illusie}.

Monoids are assumed here to be commutative with a unit element.
The group  of fractions of a monoid $P$ is denoted $P^{gp}$. We call a monoid $P$ integral if the canonical morphism $P \to P^{gp}$ is injective. We say that a monoid is fine if it is integral and of finite type. We say that a monoid $P$ is saturated if it is integral and satisfies the condition that for any $a \in P^{gp}$, if there exists $n \geq 1$ such that  $a^n \in P$, then $a \in P$. \\

A \textbf{pre-logarithmic} structure on a scheme $X$ is a sheaf of monoids $M$ on the étale site $X_{\textrm{é}t}$, endowed with a homomorphism of sheaves of monoids
$$ \alpha: M \to \mathcal{O}_X$$
where $\mathcal{O}_X$ is regarded as a monoid for the multiplicative law. A pre-log structure $M$ is called a \textbf{log} structure if $\alpha$ induces an isomorphism
$$ \alpha^{-1}(\mathcal{O}_X^{\times}) \xrightarrow{\ \simeq\ } \mathcal{O}_X^{\times}. $$

A log scheme is a scheme endowed with a log structure.  For a log scheme $X$, we shall denote the log structure of $X$ by $M_X$. A morphism of log schemes is defined in a natural way. \\

If $X$ is a scheme, the inclusion $\mathcal{O}_X^{*} \subset \mathcal{O}_X$ defines a log structure over $X$ that is called the trivial log structure. Therefore, the category of schemes can be identified with a full subcategory of the category of log schemes. More precisely, the functor of inclusion $X \to (X, \mathcal{O}_X^{*} \subset \mathcal{O}_X)$ is the right adjoint of the forgetful functor $(Y, M_Y \to \mathcal{O}_Y) \to Y$ from the category of log schemes to that of schemes. If $X$ is a log scheme, the largest Zariski open subset of $X$ (possibly empty) on which the log structure is trivial is called the \textbf{open of triviality} of $X$. \\

 We say a log scheme $X$ is fine (resp. fine and saturated) if the following condition is satisfied: Zariski locally on $X$, there exists a fine
(resp. a fine and saturated) monoid $P$ and a homomorphism $\alpha: P \to \mathcal{O}_X$ such that $M_X$ is isomorphic to
the log structure associated to the constant sheaf $P$ on $X$ regarded as a pre-log structure
with respect to $\alpha$.

\noindent Now, we recall the definition of a log structure defined by a divisor. This is the main example of log structure which we will consider in this article:

\begin{exam} \label{logdiv}
Let $X$ be a noetherian and regular scheme, and let $j: U \to X$ be an open subset whose complementary is a divisor $D$ over $X$. Then the inclusion 
$$ \mathcal{O}_X \cap j_{*}\mathcal{O}_U^{*} \to \mathcal{O}_X$$
defines a fine and saturated log structure on $X$, which we call the log structure defined by $D$. It is clear from the definition that $U$ is the open of triviality of this log structure.\\
For example, $\Spec(R)$ can be seen as a log scheme with the log structure induced by $\Spec(k)$, seen as a divisor. It is called the canonical log structure on $\Spec(R)$.
\end{exam}

In this paper, we shall endow the category of fine and saturated log schemes with the \textbf{Kummer log flat} topology (denoted sometimes by klf to simplify). We refer to \cite{kato} or \cite[\S 2.2]{Gill2} for the definition of this Grothendieck topology. A torsor defined with respect to this topology will be called a \textbf{logarithmic torsor} (or a log torsor). If $X$ is a log scheme, $G$ a group scheme over $X$, we denote by $H_{klf}^1(X,G)$ the first cohomology group that classifies $G$-logarithmic torsors over $X$. Log torsors in this paper are defined with respect to this topology. Moreover, a Kummer log flat cover of a scheme endowed with the trivial log structure is just a cover for the fppf topology. So, in this paper, the category of schemes is endowed with the fppf topology.

\begin{exam}($\mathbb{G}_m$-log torsors) \label{exam:logdiv} \\
If $X$ is a regular scheme, we write $\Div(X)$ for the group of Cartier divisors over $X$ and $\DivPrinc(X)$ for the subgroup of principal divisors. We recall that $H_{fppf}^1(X,\mathbb{G}_m)$ is isomorphic to \\ $\Div(X)/ \DivPrinc(X)$.  Now, if we assume that $X$ is endowed with a logarithmic structure induced by a divisor $D$ over $X$, then one has a similar description for $H_{klf}^1(X,\mathbb{G}_m)$. Indeed, let us call \textbf{group of divisors with rational coefficients over $D$}, and denote by $\DivRat(X,D)$ the subgroup of $\Div(X) \otimes_{\Z} \Q$ formed by the divisors over $X$ whose restriction to $U$ has integral coefficients. Then we have a canonical isomorphism
$$
\DivRat(X, D)/ \DivPrinc(X) \xrightarrow{\ \simeq\ } \ H^1_{klf}(X, \mathbb{G}_m)
$$
(see \cite[Theorem 3.1.3]{Gill1}).
\end{exam}

\subsection{Description of torsors via the Picard functor}
\subsubsection{The classical case}
If $\mathcal{G}$ is any commutative finite flat group scheme over some base $S$, $\mathcal{G}^D$ will denote its Cartier dual, namely, $\mathcal{G}^D:= \Hom_S(\mathcal{G},\mathbb{G}_{m,S})$.\\
The following isomorphism is due to Raynaud:
\begin{thm}\cite[Proposition 6.2.1]{Raynaud}\label{raynaud}
Let $f:\mathcal{X} \to S$ be a proper flat morphism of finite type, $\mathcal{G}$ an $S$-group scheme which is commutative finite and flat. Assume that $f_*\mathcal{O}_{\mathcal{X}}=\mathcal{O}_{S}$. Then we have a canonical isomorphism 
$$R^1_{fppf}f_*\mathcal{G}_{\mathcal{X}}^D \xrightarrow{\simeq} \underline{\Hom}(\mathcal{G}, \Pic_{\mathcal{X}/S}) $$
where $\Pic_{\mathcal{X}/S}$ is the relative Picard functor of $\mathcal{X}$ over $S$.
\end{thm}

\begin{lem}\label{R^1f_*}
We keep the assumptions of Theorem \ref{raynaud} and we assume in addition that $f$ has a section. We have an isomorphism: 
$$H^0(S,R_{fppf}^1f_*\mathcal{G}_{\mathcal{X}}^D) \simeq H_{fppf}^1(\mathcal{X},\mathcal{G}^D)/H_{fppf}^1(S,\mathcal{G}^D). $$
\end{lem}

\begin{proof}
Let us write the Leray sequence associated to $f$ and $\mathcal{G}_{\mathcal{X}}^D$:

\begin{small}
$0 \to H^1_{fppf}(S, f_*\mathcal{G}^D_{\mathcal{X}}) \to H^1_{fppf}(\mathcal{X},\mathcal{G}^D) \to H^0(S,R^1_{fppf}f_*\mathcal{G}^D_{\mathcal{X}}) \to H^2_{fppf}(S, f_*\mathcal{G}^D_{\mathcal{X}}) \xrightarrow{\delta} H^2_{fppf}(\mathcal{X},\mathcal{G}^D).$
\end{small}
\newline
 We have that 
 \begin{align*}
 f_* \underline{\Hom}(\mathcal{G}_{\mathcal{X}},\mathbb{G}_{m,\mathcal{X}}) &= f_*\underline{\Hom}(f^*\mathcal{G},\mathbb{G}_{m,\mathcal{X}})\\
 &=\underline{\Hom}(\mathcal{G},f_*\mathbb{G}_{m,\mathcal{X}})\\
 &=\underline{\Hom}(\mathcal{G},\mathbb{G}_{m,S})
 \end{align*}
 where the last equality follows from the assumption $f_*\mathcal{O}_{\mathcal{X}}=\mathcal{O}_{S}$. Hence, $f_*\mathcal{G}_{\mathcal{X}}^D= \mathcal{G}^D$.\\
In addition, $\delta$ is injective since $f$ has a section, and the exact sequence above becomes:
 
\begin{equation*}
0 \to H^1_{fppf}(S, \mathcal{G}^D) \to H^1_{fppf}(\mathcal{X},\mathcal{G}^D) \to H^0(S,R^1_{fppf}f_*\mathcal{G}^D_{\mathcal{X}}) \to 0.
\end{equation*}
 \end{proof}

\begin{cor}\label{Pointed}
Let $f:\mathcal{X} \to S$ be a proper flat morphism of finite type with a section, $\mathcal{G}$ a commutative finite flat $S$-group scheme. Assume that $f_*\mathcal{O}_{\mathcal{X}}=\mathcal{O}_S$. Then we have a canonical isomorphism 
$$H_{fppf}^1(\mathcal{X},\mathcal{G}^D)/H_{fppf}^1(S,\mathcal{G}^D) \xrightarrow{\simeq} {\Hom}(\mathcal{G}, \Pic_{\mathcal{X}/S}). $$
Moreover, if $\mathcal{X}$ has relative dimension $1$, then
$${\Hom}(\mathcal{G}, \Pic_{\mathcal{X}/S}) = {\Hom}(\mathcal{G}, \Pic^0_{\mathcal{X}/S}) $$

 where $\Pic^0_{\mathcal{X}/S} \subset \Pic_{\mathcal{X}/S}$ is the identity component.
\end{cor}

\begin{proof}
The first part is deduced from Theorem \ref{raynaud} and Lemma \ref{R^1f_*}.\\
As for the second part, since $\mathcal{G}$ is finite, it is torsion, hence the morphism $\mathcal{G} \to \Pic_{\mathcal{X}/S}$ factors through the functor $\Pic^{\tau}_{\mathcal{X}/S}$, where  $\Pic^{\tau}_{\mathcal{X}/S}:= \cup_{n} n^{-1} (\Pic^0_{\mathcal{X}/S})$, with $n$ being the multiplication in $\Pic_{\mathcal{X}/S}$. Finally, as noted in \cite[\S 8.0]{Raynaud}, in the case of relative curves,  $\Pic^{\tau}_{\mathcal{X}/S}=\Pic^0_{\mathcal{X}/S}$.
\end{proof}

\begin{rema}\label{pointedtor}
 \normalfont
Let $f:\mathcal{X} \to S$ be a proper flat morphism of finite type, endowed with a section which is given by some $S$-point $\mathcal{Q}$. Let $\mathcal{G}$ be a commutative finite flat $S$-group scheme. We call a \textbf{pointed} fppf $\mathcal{G}$-torsor over $\mathcal{X}$ relative to $\mathcal{Q}$ a couple $(g: \mathcal{Y} \to \mathcal{X}, \mathcal{P})$ where $\mathcal{Y}$ is an fppf $\mathcal{G}$-torsor over $\mathcal{C}$, and $\mathcal{P}$ an $S$-section over $\mathcal{Y}$ whose image by $g$ is $\mathcal{Q}$. Equivalently, an fppf $\mathcal{G}$-torsor $f: \mathcal{Y} \to \mathcal{X}$ is pointed (relative to $\mathcal{Q}$) if its restriction to $\mathcal{Q}$ is trivial.

We denote by $H^1_{fppf}(\mathcal{X},\mathcal{Q},\mathcal{G})$ the cohomology group that classifies isomorphism classes of pointed fppf $\mathcal{G}$-torsors over $\mathcal{X}$ (relative to $\mathcal{Q}$). We have an exact sequence
\[
\begin{tikzcd}
0 \to H^1_{fppf}(\mathcal{X},\mathcal{Q},\mathcal{G}) \arrow{r} & H^1_{fppf}(\mathcal{X},\mathcal{G}) \arrow{r}{\mathcal{Q}^*} & H^1_{fppf}(S,\mathcal{G}) \to 0
\end{tikzcd}
\]
where the exactness on the left is by definition of $H^1_{fppf}(\mathcal{X},\mathcal{Q},\mathcal{G})$,  and the exactness on the right is because $\mathcal{Q}^*\circ b=\mathrm{id}$, where $b:H^1_{fppf}(S,\mathcal{G}) \to H^1_{fppf}(\mathcal{X},\mathcal{G})$ is the base-change map. For the same reason, $b$ is injective and the exact sequence above splits, yielding an isomorphism
\begin{equation}\label{Pointcouple}
  H^1_{fppf}(\mathcal{X},\mathcal{Q},\mathcal{G})    \simeq  H^1_{fppf}(\mathcal{X},\mathcal{G})/H^1_{fppf}(S,\mathcal{G}).
\end{equation}
Hence, if we assume moreover that $f_*\mathcal{O}_{\mathcal{X}}=\mathcal{O}_S$, we deduce from Corollary \ref{Pointed} the isomorphism 
  \begin{equation}\label{PointedGPic}
  H^1_{fppf}(\mathcal{X},\mathcal{Q},\mathcal{G}^D)\simeq \Hom(\mathcal{G}, \Pic_{\mathcal{X}/S}).
  \end{equation}
In particular, we recover the bijective correspondence mentioned in (\ref{tor-mor}) in the case where $S=\Spec(K)$.
 \end{rema}

\subsubsection{The logarithmic case.}
Raynaud's proof of Theorem \ref{raynaud} relies on the interpretation of the relative Picard functor as the fppf sheaf $R^1_{fppf}f_*\mathbb{G}_m$; we will see that in the case where $S=\Spec(R)$, it admits an analog in the log setting. The log Picard functor has first been introduced by Kajiwara in \cite{Kajiwara} for log curves without self intersection over a field. A relative version of this functor for semistable families with fibers of any dimension has been studied by Olsson in \cite{Olsson}. For the log structure, he used what he calls the special log structure that is related to the semistability. In \cite{Alb}, Bellardini provided a comparison map between the relative log Picard functor defined using the special log structure, and the one defined using the canonical log structure on $R$ (cf. Example \ref{logdiv}). When $\mathcal{X} \to \Spec(R)$ is a semistable curve endowed with the divisorial log structure, he proved that the log Picard functor of $\mathcal{X}$ coincides with the Néron model of $\Pic_{\mathcal{X}_K/K}$.  In this paper, we define the log Picard functor using the canonical log structure of $R$, in the same way it is defined in \cite{Alb}.\\

In this section, $C$ denotes a smooth projective and geometrically connected curve over $K$, with a rational point $Q$. Let $f: \mathcal{C} \to \Spec(R)$ denote a regular model of $C$, i.e. an integral projective flat and regular $R$-scheme, with generic fiber $C$. The existence of such models is proved in \cite[ \S 8.3.4 , Corollary 3.51]{Liu}. We view $\mathcal{C}$ as a log scheme via the log structure induced by its special fiber seen as a divisor (cf. Example \ref{logdiv}). 
Finally, given the properness of $\mathcal{C}$, the $K$-section $Q$ extends uniquely into an $R$-section $\mathcal{Q}$ over $\mathcal{C}$, that we may see as a log section.

\begin{lem}\label{f_*}
$$f_*M^{gp}_{\mathcal{C}}=M^{gp}_{R}. $$
\end{lem}
\begin{proof}
Let $j: C \to \mathcal{C}$ and $j_0: \Spec(K) \to \Spec(R)$ be the canonical open immersions. The log structure on $\mathcal{C}$ is the direct image of the trivial log structure on $C$, hence it follows from the universal property of $M_{\mathcal{C}}^{gp}$ that $M_{\mathcal{C}}^{gp}=j_*M_C^{gp}=j_*\mathcal{O}_C^*$. In the same way, we have $M_{R}^{gp}=j_{0,*}M_{K}^{gp}=j_{0,*}\mathcal{O}_{K}^*$.\\
 From the commutative diagram
\[ \begin{tikzcd}
C \arrow{r}{f_K} \arrow[swap]{d}{j} & \Spec(K) \arrow{d}{j_0} \\%
\mathcal{C} \arrow{r}{f}& \Spec(R)
\end{tikzcd}
\]
we deduce that:
\begin{align*}
    (f \circ j)_*(\mathcal{O}_C^*) & = (j_0 \circ f_{K})_*(\mathcal{O}_C^*).
\end{align*}
But, 
\begin{align*}
    (j_0 \circ f_{K})_*(\mathcal{O}_C^*) &= j_{0,*}(f_{K,*}\mathcal{O}_C^*)\\
    &= j_{0,*}\mathcal{O}_{K}^*\\
    &= M_{R}^{gp}
\end{align*}
where the second equality comes from the fact that $f_{K,*}(\mathcal{O}_C)=\mathcal{O}_{K}$, which is a consequence of the fact that $C$ has a section over $K$.
On the other hand, 
\begin{align*}
    (f \circ j)_*(\mathcal{O}_C^*) &=f_*(j_*\mathcal{O}_C^*)\\
    &=f_*M_{\mathcal{C}}^{gp}.
\end{align*}
This ends the proof.
\end{proof}

\begin{defi}\label{Piclog}
We recall that $\Spec(R)$ is seen as a log scheme via the log structure induced by $\Spec(k)$. Let $(Sch/R)$ denote the category of schemes over $R$. It can be seen as a (full) subcategory of the category $(fs/R)$ of fine and saturated log schemes over $R$ as follows: given a morphism of schemes $T \to \Spec(R)$, we put on $T$ the inverse log structure of that on $\Spec(R)$. The topology on $(Sch/R)$ induced by the klf topology is the fppf one.\\
\begin{enumerate}
\item Recall the definition of the following functor:
\begin{center}
\begin{align*}
       \mathbb{G}_{m,log,R}: (fs/R) & \to (Ab)\\
    T & \mapsto \Gamma(T, M_{T}^{gp}).
    \end{align*}
\end{center}
This is a sheaf for the Kummer log flat topology \cite[Theorem 3.2]{kato}. 
\item Consider the following functor
\begin{align*}
    (Sch/R) &\to (Sets)\\
    T  & \mapsto \{\mathbb{G}_{m,log,\mathcal{C}}-log~torsors~on~\mathcal{C}_{T} \}.
\end{align*}

The \textbf{log Picard functor}, denoted by $\Pic^{log}_{\mathcal{C}/R}$, is defined to be the fppf sheafification on $(Sch/R)$ of the previous functor. It can also be defined by the formula 
$$\Pic_{\mathcal{C}/R}^{log}(T) = H^0(T, R^1f_*\mathbb{G}_{m,log,\mathcal{C}}) $$
where $R^1f_*\mathbb{G}_{m,log,\mathcal{C}}$ is computed using the Kummer log flat topology.
\end{enumerate}

\end{defi}

\begin{thm}\label{raynaudlog}
 Let $\mathcal{G}$ be a commutative finite flat group scheme over $R$. We have a canonical isomorphism:
 $$R^1f_*(\mathcal{G}_{\mathcal{C}}^D) \xrightarrow{\simeq} \underline{\Hom}(\mathcal{G}, \Pic^{log}_{\mathcal{C}/R}). $$
 \end{thm}
 
 \begin{proof}
 We shall check that the same arguments as for Theorem \ref{raynaud} can be transported to the log setting.\\
We have an exact sequence on the Kummer log flat site (cf. \cite[exact sequence 2.3.2]{Gill1}):
$$0 \to  \mathbb{G}_{m,\mathcal{C}} \to \mathbb{G}_{m,log,\mathcal{C}} \to (\mathbb{G}_{m,log,\mathcal{C}}/\mathbb{G}_{m,\mathcal{C}})^{klf} \to 0,$$
where the quotient is computed in the Kummer log flat site. We deduce a long exact sequence 
\begin{equation}\label{Homm}
    0 \to \underline{\Hom}(\mathcal{G}_{\mathcal{C}},\mathbb{G}_{m,\mathcal{C}}) \to \underline{\Hom}(\mathcal{G}_{\mathcal{C}},\mathbb{G}_{m,log,\mathcal{C}}) \to \underline{\Hom}(\mathcal{G}_{\mathcal{C}},(\mathbb{G}_{m,log,\mathcal{C}}/\mathbb{G}_{m,\mathcal{C}})^{klf}).
\end{equation}

It is proved in \cite[Lemma 2.3.1]{Gill1} that the multiplication by any integer $n$ is an automorphism in $(\mathbb{G}_{m,log,\mathcal{C}}/\mathbb{G}_{m,\mathcal{C}})^{klf}$, hence the quotient $(\mathbb{G}_{m,log,\mathcal{C}}/\mathbb{G}_{m,\mathcal{C}})^{klf}$ has no torsion points. Since $\mathcal{G}_{\mathcal{C}}$ is finite, the last term of the exact sequence (\ref{Homm}) is then trivial and hence we have
\begin{equation}\label{Hom}
\underline{\Hom}(\mathcal{G}_{\mathcal{C}},\mathbb{G}_{m,log,\mathcal{C}})=\underline{\Hom}(\mathcal{G}_{\mathcal{C}},\mathbb{G}_{m,\mathcal{C}})=\mathcal{G}_{\mathcal{C}}^D.
\end{equation}

By the same argument we prove the first equalities in each line
\begin{align}\label{Ext}
\begin{split}
\underline{\Ext}_{klf}^1(\mathcal{G}_{\mathcal{C}},\mathbb{G}_{m,log,\mathcal{C}}) & =\underline{\Ext}_{klf}^1(\mathcal{G}_{\mathcal{C}},\mathbb{G}_{m,\mathcal{C}}) =\{0\} \\
\underline{\Ext}_{klf}^1(\mathcal{G},\mathbb{G}_{m,log}) & =\underline{\Ext}_{klf}^1(\mathcal{G},\mathbb{G}_{m}) = \{0\}
\end{split}
\end{align}
while the second ones follow from \cite[Theorem 4.1]{Gill2}.\\

We can now write the log version of the exact sequence appearing in the proof of Theorem \ref{raynaud}. Consider the following functor:

$$F \mapsto H(F):=f_*\underline{\Hom}(f^*\mathcal{G},F)= \underline{\Hom}(\mathcal{G},f_* F),$$
for some sheaf $F$ on the log flat site of $\mathcal{C}$. 	The derived functor $R^iH$ is the result of two spectral sequences:
\begin{center}
\begin{align*}
 R^qf_*\underline{\Ext}^p(f^*\mathcal{G},F) & \Rightarrow R^{p+q}H(F) \\
\underline{\Ext}^q(\mathcal{G},R^pf_*F) & \Rightarrow R^{p+q}H(F). 
\end{align*}
\end{center}
We take $F$ to be the sheaf $\mathbb{G}_{m,log,\mathcal{C}}$. We obtain a commutative diagram with exact lines: 
\begin{center}
\small
\xymatrix@C=1em{
  0 \ar[r] & R^1f_*(\underline{\Hom}(\mathcal{G}_{\mathcal{C}},\mathbb{G}_{m,log,\mathcal{C}}))  \ar[d]_-{} \ar[r]^-{} & R^1H(\mathbb{G}_{m,log,\mathcal{C}}) \ar[d]_-{} \ar[r]^-{} & f_*\underline{\Ext}_{klf}^1(\mathcal{G}_{\mathcal{C}},\mathbb{G}_{m,log,\mathcal{C}})  \ar[d]^-{}  \\
  0 \ar[r] & \underline{\Ext}^1_{klf}(\mathcal{G},f_*\mathbb{G}_{m,log,\mathcal{C}})  \ar[r]_-{} & R^1H(\mathbb{G}_{m,log,\mathcal{C}}) \ar[r]_-{} & \underline{\Hom}(\mathcal{G},R^1f_*\mathbb{G}_{m,log,\mathcal{C}}) \ar[r] & \underline{\Ext}^2_{klf}(\mathcal{G},f_*\mathbb{G}_{m,log,\mathcal{C}}) \ar[r] & R^2H(\mathbb{G}_{m,log,\mathcal{C}}).
}
\end{center}
Now, using (\ref{Hom}), (\ref{Ext}) and Lemma \ref{f_*}, the diagram becomes 

\begin{center}
\small
\xymatrix@C=1em{
  0 \ar[r] & R^1f_*(\mathcal{G}^D_{\mathcal{C}})  \ar[d]_-{} \ar[r]^-{=} & R^1H(\mathbb{G}_{m,log,\mathcal{C}}) \ar[d]_-{=} \ar[r]^{} & 0 \\
& 0 \ar[r]_-{} & R^1H(\mathbb{G}_{m,log,\mathcal{C}}) \ar[r]_-{} & \underline{\Hom}(\mathcal{G},R^1f_*\mathbb{G}_{m,log,\mathcal{C}}) \ar[r] & \underline{\Ext}^2_{klf}(\mathcal{G},f_*\mathbb{G}_{m,log,\mathcal{C}}) \ar[r] & R^2H(\mathbb{G}_{m,log,\mathcal{C}}).
}
\end{center}

We deduce the following exact sequence

$$0 \to R^1f_{*}(\mathcal{G}_{\mathcal{C}}^D) \to \underline{\Hom}(\mathcal{G},R^1f_{*}\mathbb{G}_{m,log,\mathcal{C}}) \to \underline{\Ext}^2_{klf}(\mathcal{G},f_*\mathbb{G}_{m,log, \mathcal{C}}) \xrightarrow{\gamma} R^2H(\mathbb{G}_{m,log,\mathcal{C}}) $$
with $\gamma$ being injective with the same arguments as in \cite[Proposition 6.2.1]{Raynaud} (after noticing that $\mathcal{C}$ has an $R$-section), which ends the proof.
 \end{proof}

\begin{cor}\label{pointedlog}
Let $\mathcal{G}$ be a commutative finite flat group scheme over $R$. We have a canonical isomorphism:
 $$H^1_{klf}(\mathcal{C},\mathcal{G}_{\mathcal{C}}^D)/H^1_{klf}(R,\mathcal{G}^D) \xrightarrow{\simeq} \Hom(\mathcal{G}, \Pic^{log}_{\mathcal{C}/R}). $$
\end{cor} 
 
 \begin{proof}
 The proof is similar to the proof of Lemma \ref{R^1f_*} and Corollary \ref{Pointed}. 
 \end{proof}

\begin{rema}
\normalfont
View $\mathcal{Q}$ as a log point. We call a \textbf{pointed} $\mathcal{G}$-log torsor over $\mathcal{C}$ (relative to $\mathcal{Q}$) a couple ($g: \mathcal{Y} \to \mathcal{C}$, $\mathcal{P}$) where $\mathcal{Y}$ is a $\mathcal{G}$-log torsor over $\mathcal{C}$, and $\mathcal{P}$ is an $R$-log section over $\mathcal{Y}$ whose image by $g$ is $\mathcal{Q}$. Equivalently, a $\mathcal{G}$-log torsor over $\mathcal{C}$ is pointed (relative to $\mathcal{Q}$) if its restricton to $\mathcal{Q}$ is trivial. We denote by $H^1_{klf}(\mathcal{C},\mathcal{Q},\mathcal{G})$ the cohomology group that classifies isomorphism classes of pointed $\mathcal{G}$-log torsors over $\mathcal{C}$ (relative to $\mathcal{Q}$). We prove similarly as for (\ref{PointedGPic}) the isomorphism

  \begin{equation}\label{PointedlogGPic}
  H^1_{klf}(\mathcal{C},\mathcal{Q},\mathcal{G}^D)\simeq \Hom(\mathcal{G}, \Pic^{log}_{\mathcal{C}/R}).
  \end{equation}
\end{rema}

\subsection{Extension of torsors} 
 The isomorphism (\ref{PointedlogGPic}) shows that extending torsors from $C$ to $\mathcal{C}$ reduces to the extension of some group functors and morphisms between them. The aim of this section is to replace the log Picard functor by a representable object, namely the Néron model of the Jacobian of the curve. This gives the so-called \textit{Néron-log} torsors and induces a sufficient condition for the initial torsor to extend, that moreover can be checked in practice.\\ 

 Keeping the same notation as in the previous section, if $J=\Pic^0_{C/K}$ is the Jacobian of $C$, $\mathcal{J}$ will denote its Néron model over $R$ (cf. \cite[\S 1.3, Corollary 2]{BLR}).
 
Since $C$ has a section, it follows from \cite[\S 8.1, Proposition 4]{BLR} that we have an isomorphism:
\begin{equation*}\label{secPic}
 \Pic_{C/K}(T) \simeq \Pic(C \times_K T)/ \Pic(T).
\end{equation*}

\noindent Given that the curve $C$ is a proper scheme over a field, $\Pic_{C/K}$ is representable by a $K$-group scheme, called the Picard scheme, and we denote it in the same way \cite[\S 8.2. Theorem 3]{BLR}. 

\begin{prop} \label{JPic}
The closed immersion $J \hookrightarrow \Pic_{C/K}$ in the generic fiber extends uniquely to a morphism $\mathcal{J} \to \Pic_{\mathcal{C}/R}^{log}$. 
\end{prop}

\begin{proof}
Let $T$ be an $R$-scheme endowed with the inverse log structure of that of $\Spec(R)$. The Leray spectral sequence associated to $f: \mathcal{C} \to \Spec(R)$ and $\mathbb{G}_{m,log,\mathcal{C}}$ gives an exact sequence: \\

\begin{center}
$0 \to H^1_{klf}(T,f_*\mathbb{G}_{m,log,\mathcal{C}}) \to H^1_{klf}(\mathcal{C}_T,\mathbb{G}_{m,log,\mathcal{C}}) \to H^0(T,R^1f_*\mathbb{G}_{m,log,\mathcal{C}}) \to H^2_{klf}(T,f_*\mathbb{G}_{m,log,\mathcal{C}}) \rightarrow H^2_{klf}(\mathcal{C}_T,\mathbb{G}_{m,log,\mathcal{C}}).$
\end{center}

All the log structures coming into play here are inverse image of that of $\Spec(R)$, hence it follows from Lemma \ref{f_*} that $f_*\mathbb{G}_{m,log,\mathcal{C}} = \mathbb{G}_{m,log,R}$. Therefore, the exact sequence becomes:    

\begin{center}
$0 \to H^1_{klf}(T,\mathbb{G}_{m,log,R}) \to H^1_{klf}(\mathcal{C}_T,\mathbb{G}_{m,log,\mathcal{C}}) \to H^0(T,R^1f_*\mathbb{G}_{m,log,\mathcal{C}}) \to H^2_{klf}(T,\mathbb{G}_{m,log,R}) \xrightarrow{\delta} H^2_{klf}(\mathcal{C}_T,\mathbb{G}_{m,log,\mathcal{C}})$
\end{center}

where $\delta$ is injective since $\mathcal{C}$ has a section. Therefore, the previous exact sequence implies the following one

$$
0 \to H^1_{klf}(T,\mathbb{G}_{m,log,R}) \to H^1_{klf}(\mathcal{C}_T,\mathbb{G}_{m,log,\mathcal{C}}) \to H^0(T,R^1f_*\mathbb{G}_{m,log,\mathcal{C}}) \to 0.
$$

Now, assume that $T$ is smooth over $R$. Then, $\mathcal{C}_{T}$ is $\mathcal{C}$-smooth, and hence is regular. Therefore, using \cite[Proposition 2.2.6]{Gill1}, we have that $H^1_{klf}(\mathcal{C}_T,\mathbb{G}_{m,log,\mathcal{C}})=\Pic(C \times T_K)$ and $H^1_{klf}(T,\mathbb{G}_{m,log,R})=\Pic(T_K)$.  Hence, for any $R$-smooth scheme $T$, we have 
$$\Pic_{C/K}(T_K)=\Pic(C \times T_K)/\Pic(T_K) = H^0(T,R^1f_*\mathbb{G}_{m,log,\mathcal{C}})=\Pic^{log}_{\mathcal{C}/R}(T).$$
In particular, since $\mathcal{J}$ is smooth, we have 
$$\Pic_{C/K}(J)=\Pic^{log}_{\mathcal{C}/R}(\mathcal{J}). $$
Therefore, the injective canonical map $J \hookrightarrow \Pic_{C/K}$ extends uniquely into a morphism $\mathcal{J} \to \Pic^{log}_{\mathcal{C}/R}$.

\end{proof}

\begin{rema}
\normalfont
The fact that for any smooth $R$-scheme $T$ the restriction to the generic fiber $\Pic^{log}_{\mathcal{C}/R}(T)=\Pic_{C/K}(T_K)$ is bijective implies that the Néron model of $\Pic_{C/K}$ (cf. \cite[\S 10.2, Theorem 1]{BLR}) represents $\Pic^{log}_{\mathcal{C}/R}$ in the smooth site. Note that this also holds in the fppf site, i.e. the Néron model of $\Pic_{C/K}$ represents the relative Picard functor in the smooth site. This uses the fact that the relative Picard functor is separated when $\mathcal{C}$ has an $R$-section (cf. \cite[\S 9.5]{BLR} and \cite[Theorem 8.2 (i)]{BLR}).
\end{rema}

\begin{defi}
\begin{enumerate}
\item If $G$ is a commutative finite group scheme over $K$, we call a model of $G$ over $R$ a finite flat group scheme $\mathcal{G}$ over $R$, whose generic fiber is isomorphic to $G$.
\item We call a pointed $\mathcal{G}$-log torsor over $\mathcal{C}$ whose associated morphism $\mathcal{G}^D \to \Pic^{log}_{\mathcal{C}/R}$ factors through $\mathcal{J}$ (via the map of Proposition \ref{JPic}), a \textbf{Néron-log} torsor. Given the separatedness of $\mathcal{J}$, the factorization is unique when it exists.
\end{enumerate}
\end{defi}

\begin{lem}\cite[Proposition 3.6]{Gill2}
If $\mathcal{G}$ is a model of $G$, then the restriction morphism
$$H^1_{klf}(\mathcal{C},\mathcal{G}) \to H^1_{fppf}(C,G)$$
is injective.
\end{lem}

\begin{cor}\label{uni}
If $\mathcal{G}$ is a model of $G$, then the restriction morphism
$$H^1_{klf}(\mathcal{C},\mathcal{Q},\mathcal{G}) \to H^1_{fppf}(C,Q,G)$$
is injective.
\end{cor}

\begin{cor}\label{G^D -> J}
Let $G$ be a commutative finite $K$-group scheme and let $Y \to C$ be an fppf pointed $G$-torsor. If the associated morphism $h_{Y}: G^D \to J$ extends to a morphism $\mathcal{G}^D \to \mathcal{J}$ \footnote{Once the model $\mathcal{G}$ is fixed, if the morphism $h_Y$ extends to $\mathcal{G}^D \to \mathcal{J}$, this extension is unique because $\mathcal{J}$ is separated.} for some $R$-model $\mathcal{G}$ of $G$, then the $G$-torsor $Y \to C$ extends uniquely to a Néron-log (pointed) $\mathcal{G}$-torsor over $\mathcal{C}$.
\end{cor}

\begin{proof}
The unicity of extension once $\mathcal{G}$ is fixed follows from Corollary \ref{uni}. The existence follows from (\ref{PointedlogGPic}) and Proposition \ref{JPic}.
\end{proof}

\begin{prop}\label{G^D->J^0}
Let $G$ be a commutative finite $K$-group scheme and let $Y \to C$ be an fppf pointed $G$-torsor. Assume that the associated morphism $h_{Y}:G^D \to J$ extends to a morphism $\mathcal{G}^D \to \mathcal{J}$ for some model $\mathcal{G}$ of $G$, so that we have a Néron-log extension of the torsor. Then this extended torsor comes from an fppf one if and only if the morphism $\mathcal{G}^D \to \mathcal{J}$ factors through $\mathcal{J}^0$, the identity component of $\mathcal{J}$.
\end{prop}

\begin{proof}
Using (\ref{PointedGPic}), it suffices to justify why $\Pic_{\mathcal{C}/R}^0=\mathcal{J}^0$. Since $f: \mathcal{C} \to \Spec(R)$ has a section, the $gcd$ of the geometric multiplicities of the irreducible components of the special fiber of $\mathcal{C}$ is equal to $1$, and the result follows from \tsb{\cite[Proposition 4.2.1 (1) and Theorem 8.2 (i)]{Raynaud}}. 
\end{proof}

\begin{rema} \label{premier}
\normalfont
Let $G$ be a commutative finite $K$-group scheme and let $\Phi$ denote the group of components of the special fiber $\mathcal{J}_k$ of $\mathcal{J}$. Given a pointed $G$-torsor $Y \to C$, assume that the morphism $h_Y: G^D \to J$ extends to a morphism $\mathcal{G}^D \to \mathcal{J}$, so that we have a Néron-log extension of the torsor over $\mathcal{C}$. Then if $\# \mathcal{G} \bigwedge \# \Phi =1$, the extended torsor is fppf. Indeed, by definition, $\Phi=\mathcal{J}_k/\mathcal{J}_k^0$. So, we have an exact sequence of fppf-sheaves of groups:

$$ 0 \to \mathcal{J}_k^0 \to \mathcal{J}_k \to \Phi \to 0$$
from which, if $\mathcal{G}_k$ denotes the special fiber of $\mathcal{G}$, we deduce the long exact sequence of fppf sheaves of groups

$$ 0\to \Hom_k(\mathcal{G}_k^D,\mathcal{J}_k^0) \to \Hom_k(\mathcal{G}_k^D,\mathcal{J}_k) \to \Hom_k(\mathcal{G}_k^D, \Phi) \to \cdots. $$ 
Since $\# \mathcal{G} \bigwedge \# \Phi =1$, $\Hom_R(\mathcal{G}_k^D, \Phi) =  \{0\}$ and the result follows from Proposition \ref{G^D->J^0}.
\end{rema}

\section{Obstruction of extension into fppf torsors}\label{section2}
As in the previous section, $C$ denotes a smooth projective and geometrically connected $K$-curve, with a rational point $Q$, which extends uniquely into an $R$-section $\mathcal{Q}$ over some fixed $R$-regular model $\mathcal{C}$ of $C$. We denote by $J$ the Jacobian of $C$ and by $\mathcal{J}$ its Néron model. We fix a commutative finite $K$-group scheme $G$.\\
We give in this section a description of the Néron-log torsors in terms of the log Poincaré extension. In addition, we compute the obstruction for a Néron-log torsor to come from an fppf one. In particular, we will see that this obstruction can be expressed using the obstruction for the Poincaré extension of $J_C$ by $\mathbb{G}_{m,C}$ to extend into an fppf one (cf. \cite[Exposé VIII; Remark 7.2]{Groth1}). To do this, we first need to relate the isomorphism (\ref{PointedGPic}) to the Poincaré extension.

\subsection{Preliminaries}

\noindent For each $R$-scheme $T$, we denote by $\mathcal{Q}_{T}: T \to \mathcal{C} \times T$ the section of $\mathcal{C}\times T$ induced by the $R$-point $\mathcal{Q}\in \mathcal{C}(R)$. A \textbf{rigidified line bundle} over $\mathcal{C} \times T$ is a couple $(\mathcal{L},\alpha)$ where $\mathcal{L}$ is a line bundle over $\mathcal{C} \times T$ and $\alpha$ is a trivialization $\mathcal{O}_T \simeq \mathcal{Q}_{T}^{*}\mathcal{L}$.\\

The functor $\Pic^{rig}(\mathcal{C}, .): (Sch/R) \to (Sets)$, which associates to each $R$-scheme $T$
the set $\Pic^{rig}(\mathcal{C}, T)$ of isomorphism classes of line bundles on $\mathcal{C} \times T$, which are rigidified along the
section $\mathcal{Q}_{T}$, is in fact a sheaf with respect to the fppf topology. We denote by $\Pic^{0,rig}(\mathcal{C},.)$ the subfunctor which associates to each $R$-scheme $T$ the set of isomorphism classes of rigidified line bundles with relative degree zero, i.e. line bundles whose restrictions to each geometric fiber of $\mathcal{C}\times T \to T$ have degree zero.\\ 
Given that $\mathcal{C}$ has a section, it is shown in \cite[\S 8.1]{BLR}, in the discussion that follows Proposition 4, that the sheaf $\Pic^{rig}(\mathcal{C},.)$ is canonically isomorphic to the relative Picard functor $\Pic_{\mathcal{C}/R}$. For any $R$-scheme $T$, forgetting the trivialization induces an isomorphism of groups
\begin{equation*}
\Pic^{rig}(\mathcal{C},T) \xrightarrow{\simeq} \Pic_{\mathcal{C}/R}(T). 
\end{equation*}
In fact, if one replaces $G$ by $\mathbb{G}_m$ and $T$ by $\mathrm{Spec}(R)$, then  this isomorphism is none other than the isomorphism (\ref{Pointcouple}).

It follows that 
\begin{equation}\label{rig0}
\Pic^{0,rig}(\mathcal{C},T) \simeq \Pic^0_{\mathcal{C}/R}(T) = \Hom_{R-\mathrm{schemes}}(T,\Pic^0_{\mathcal{C}/R})
\end{equation}
where the second isomorphism follows from \cite[\S 9.3; Corollary 13]{BLR}. 
\\
At the level of the generic fibers, if we let $T= \Pic^0_{C/K}=J$, then the identity map $\Pic^0_{C/K} \to \Pic^0_{C/K}$ corresponds to the universal line bundle $\mathcal{P}_K$ over $C \times J$, which is a rigidified line bundle of relative degree $0$. It is called the \textbf{Poincar\'e line bundle}  over $C \times J$. It satisfies the following universal property: to any rigidified line bundle $\mathcal{L}$ of degree $0$ over $C \times T$ corresponds, by definition, a unique element of $\Pic^{0,rig}(C,T)$,  thus a unique morphism of schemes $f: T \to \Pic^0_{C/K}=J$, and we have:
$$
(id_{C} \times f)^{*} \mathcal{P}_K \simeq \mathcal{L}.
$$
\hop \hop \hop

The following lemma will be used in the sequel in order to handle  group extensions via torsors. We state it in a general framework in order to cover schemes and log schemes at the same time.

\begin{lem}
\label{lem:SGA7}
Let $G$ and $H$ be two group objects in some topos. Let $pr_1,pr_2$ and $m$ be the projections and multiplication map $H\times H\to H$. Then the data of an extension $E$ of $H$ by $G$ is equivalent to the data of a $G$-torsor $E\to H$, together with an isomorphism of $G$-torsors
\begin{equation}
\label{eq:square}
m^*E\simeq pr_1^{*}E \times^G pr_2^{*}E,
\end{equation}
where $\times^G $ denotes the contracted product of $G$-torsors.
\end{lem}

\begin{proof}
This is \cite[VII, \S 1.1.6]{Groth1}.
\end{proof}

In the setting of Lemma~\ref{lem:SGA7}, we say that a $G$-torsor $E\to H$ satisfies the theorem of the square if there exists an isomorphism as in \eqref{eq:square}.

\noindent Any rigidified line bundle of degree zero on an abelian scheme verifies the theorem of the square. Therefore, if we view $J \times C$ as the abelian scheme $J_{C}$ over $C$ and if we denote by $m$, (resp. $pr_1$, $pr_2$) the multiplication (resp. the first projection, the second projection) in $J_{C}$ induced by that of $J$, we have an isomorphism of rigidified line bundles
$$
m^{*}\mathcal{P}_K \simeq pr_1^{*}\mathcal{P}_K \otimes pr_2^{*}\mathcal{P}_K.
$$

We deduce from that:

\begin{lem} \label{PoinExt1}
$\mathcal{P}_K$ has a unique underlying structure of extension of $J_{C}$ by $ \mathbb{G}_{m,C}$.
\end{lem}

\begin{proof}
The existence of an extension structure follows from Lemma~\ref{lem:SGA7}. The uniqueness follows from the fact that $m^{*}\mathcal{P}_K $ is a rigidified line bundle on the projective variety $J\times J\times C$, hence has no nontrivial automorphism.
\end{proof}

\noindent By abuse of notation, we denote this extension by $\mathcal{P}_K$ again. \hop

\noindent We have arrived at the following description of Theorem \ref{raynaud} in terms of the Poincaré extension:

\begin{lem} \label{Mon lemme}
Let $\mathcal{G}$ be a commutative finite flat $R$-group scheme. We have isomorphisms
$$ H^1_{fppf}(\mathcal{C},\mathcal{Q},\mathcal{G}) \simeq \Hom_{R}(\mathcal{G}^D,\Pic^0_{\mathcal{C}/R})\simeq \Pic^{0,rig}(\mathcal{C} \times \mathcal{G}^D)$$
where the second isomorphism is given at the level of the generic fibers by  
$$f \mapsto \mathcal{L}:=(id_C \times f)^*\mathcal{P}_K. $$
\end{lem}

\subsection{The log Poincaré extension}

Now, we shall prove that over a regular scheme endowed with the log structure defined by a divisor, extensions of smooth group schemes by $\mathbb{G}_m$ over the open subset of triviality can be extended to the logarithmic site.

\begin{prop} \label{prop ext}
Let $X$ be a regular and integral scheme, endowed with the logarithmic structure defined by a divisor with complement $U$. Let $\mathcal{H}$ be a commutative $X$-group scheme, smooth and of finite type over $X$. Then the restriction morphism
$$
\Ext_{klf}^1(\mathcal{H}, \mathbb{G}_{m,X}) \to \Ext^1_{fppf}(\mathcal{H}_U,\mathbb{G}_{m,U})
$$
is an isomorphism, where $\mathcal{H}_U:=\mathcal{H} \times_X U$.
\end{prop}

\noindent In order to prove this proposition, we shall use the following lemma.

\begin{lem} \label{extension droites log}
Let $X$ be a regular and integral scheme endowed with the logarithmic structure defined by a divisor, and let $V$ be an open subset of $X$ such that $\mathrm{codim}(X\backslash V) \geq 2$. Then the restriction map
$$
H^1_{klf}(X,\mathbb{G}_m) \to H^1_{klf}(V, \mathbb{G}_m)
$$
is an isomorphism.
\end{lem}

A brief comment on the terminology: $\mathrm{codim}(X\backslash V) \geq 2$ means that every point of $X\backslash V$ has codimension at least $2$ in $X$.

\begin{proof}[Proof of Lemma~\ref{extension droites log}]
Let $D$ be the divisor defining the log structure of $X$. Then one can write $D:= \sum_{m=1}^r D_m$, where the $D_m$'s are irreducible and reduced divisors.
Consider the following diagram, in which the lines are exact according to \cite[Corollary 3.1.4]{Gill1}, and where the vertical arrows denote the restrictions to $V$:
$$
\xymatrix{
  0 \ar[r] & H^1_{fppf}(X,\mathbb{G}_m)  \ar[d]_-{} \ar[r]^-{} & H^1_{klf}(X,\mathbb{G}_m) \ar[d]_-{} \ar[r]^-{} & \bigoplus\limits_{m=1}^r(\mathbb{Q}/ \mathbb{Z}).D_m  \ar[d]^-{} \ar[r] & 0 \\
  0 \ar[r] & H^1_{fppf}(V,\mathbb{G}_m)  \ar[r]_-{} & H^1_{klf}(V, \mathbb{G}_m) \ar[r]_-{} & \bigoplus\limits_{m=1}^r(\mathbb{Q}/ \mathbb{Z}).(D_m)_{V}  \ar[r] & 0.
}
$$

\noindent Since $X$ is regular, Weil and Cartier divisors agree on $X$. In addition, given that $X$ is normal and integral and that $\mathrm{codim}(X\backslash V)\geq 2$, any Weil divisor on $V$ extends uniquely into a Weil divisor on $X$. Therefore, the two vertical arrows in the two ends of the diagram above are bijective. It follows that the vertical arrow in the middle is bijective, too.

\end{proof}

\begin{proof}[Proof of Proposition~\ref{prop ext}]
Consider the restriction morphism
$$
\Ext_{klf}^1(\mathcal{H}, \mathbb{G}_{m,X}) \to \Ext^1_{fppf}(\mathcal{H}_U,\mathbb{G}_{m,U}).
$$ 

Let us prove that it is surjective. Let us fix an extension $E$ of $\mathcal{H}_U$ by $\mathbb{G}_{m,U}$. Let $V$ be the largest open subset of $X$ over which $E$ has a logarithmic extension. Then $V$ contains $U$, and it also contains all points of $X$ of codimension $1$ because  the extension problem has a solution over the spectrum of a discrete valuation ring by \cite[Theorem 4.1.1]{Gill1}. In other terms, we have that $\codim(X \setminus V) \geq 2$.

Now, according to Lemma~\ref{lem:SGA7}, this extension over $V$ can be viewed as a $\mathbb{G}_m$-log torsor over $\mathcal{H}_V$ that verifies the theorem of the square. Let us now apply Lemma~\ref{extension droites log} to this $\mathbb{G}_m$-log torsor on $\mathcal{H}_V$, observing, on the one hand, that $\mathcal{H}_V$ is an open subset of $\mathcal{H}$ whose complement has codimension at least $2$, and, on the other hand, that $\mathcal{H}$ being smooth over $X$, it is also regular, and the log structure induced by $X$ on $\mathcal{H}$ is also the log structure defined by a divisor on $\mathcal{H}$. Therefore, our $\mathbb{G}_m$-log torsor on $\mathcal{H}_V$  can be extended to a $\mathbb{G}_m$-log torsor on $\mathcal{H}$. This $\mathbb{G}_m$-log torsor also satisfies the theorem of the square by uniqueness of the extension over $\mathcal{H} \times_X \mathcal{H}$. According to Lemma~\ref{lem:SGA7} again, we obtain a logarithmic extension of $\mathcal{H}$ by $\mathbb{G}_{m,X}$ that extends $E$, hence the result.

As for the injectivity of the restriction map, it follows from the unicity of extension on points of $X$ of codimension $1$ (cf. \cite[Theorem 4.1.1]{Gill1}), and the unicity of extension from an open subset of $X$ whose complementary has codimension greater than $2$ to the whole $X$ (cf. Lemma~\ref{extension droites log}).

\end{proof}

\begin{cor} \label{PoinExt}
The Poincar\'e extension $\mathcal{P}_K$ from Lemma~\ref{PoinExt1} can be uniquely extended into a logarithmic extension of $\mathcal{J}_{\mathcal{C}}$ by $\mathbb{G}_{m,\mathcal{C}}$, that we denote by $\mathcal{P}^{log}$. Moreover, it is rigidified  with respect to the section induced by $\mathcal{Q}$.

\end{cor}

\begin{proof}
Since $\mathcal{J}$ is the N\'eron model of $J$ over $\Spec(R)$, it is smooth and of finite type over $\Spec(R)$. Therefore, $\mathcal{J}_{\mathcal{C}}$ is smooth and is of finite type over $\mathcal{C}$ and, according to Proposition~\ref{prop ext}, it follows that $\mathcal{P}_K$ extends uniquely into a logarithmic extension of $\mathcal{J}_{\mathcal{C}}$ by $\mathbb{G}_{m,\mathcal{C}}$, that we denote by $\mathcal{P}^{log}$.

In order to check that $\mathcal{P}^{log}$ is rigidified, consider the commutative diagram:
\[
\begin{tikzcd}
\Ext_{klf}^1(\mathcal{J} \times \mathcal{C}, \mathbb{G}_{m,\mathcal{C}}) \arrow{r}{\mathcal{Q}_{\mathcal{J}}^{*}} \arrow[]{d}{\simeq} & \Ext_{klf}^1(\mathcal{J}, \mathbb{G}_m)  \arrow[]{d}{\simeq} \\%
\Ext_{fppf}^1(J \times C, \mathbb{G}_{m,C})  \arrow{r}{Q_{J}^{*}}& \Ext_{fppf}^1(J, \mathbb{G}_m)
\end{tikzcd}
\]
in which $\mathcal{Q}_{\mathcal{J}}: \mathcal{J} \to \mathcal{J} \times \mathcal{C}$ denotes the section induced by $\mathcal{Q}$, and the vertical maps are the restrictions which are isomorphisms by Proposition~\ref{prop ext}.  Since $\mathcal{P}_K$ is rigidified along the section induced by $Q$, it implies that $Q_{J}^{*}\mathcal{P}_K$ is trivial, and then it follows from the diagram above that $\mathcal{Q}_{\mathcal{J}}^{*} \mathcal{P}^{log}$ is trivial as well.
\end{proof}
It is shown in \cite[VIII, Remark 7.2]{Groth1} that the morphism which associates to an extension its generic fiber:
$$
\Ext^1_{fppf}(\mathcal{J}^0_{\mathcal{C}}, \mathbb{G}_{m,\mathcal{C}}) \to \Ext^1_{fppf}(J_{C}, \mathbb{G}_{m,C})
$$
is an isomorphism. Therefore, we can extend $\mathcal{P}_K$ (see Lemma~\ref{PoinExt1}) into an fppf extension of $\mathcal{J}^0$ by $\mathbb{G}_m$ over $\mathcal{C}$, that we denote by $\mathcal{P}^0$. In fact, given the uniqueness of the extension, and since any fppf extension can be seen as a logarithmic one, it follows that the pull-back of $\mathcal{P}^{log}$ by the inclusion map $\mathcal{J}^0\subseteq \mathcal{J}$ is equal to $\mathcal{P}^0$, which implies for the associated line bundles that:
$$
\mathcal{P}^{log}|_{\mathcal{J}^0\times\mathcal{C}}= \mathcal{P}^0.
$$
In particular, $\mathcal{P}^0$ is rigidified.

\subsection{Obstruction of extension into fppf torsors}\label{obstr}
If $\mathcal{G}$ is a commutative finite flat $R$-group scheme, we have two canonical isomorphisms 
\begin{equation}\label{eq:2}
H^1_{klf}(\mathcal{C},\mathcal{G}) \simeq \Ext^1_{klf}(\mathcal{G}^D_{\mathcal{C}}, \mathbb{G}_{m,\mathcal{C}})
\end{equation}

\begin{equation}\label{eq:2.0}
H^1_{fppf}(\mathcal{C},\mathcal{G}) \simeq \Ext^1_{fppf}(\mathcal{G}^D_{\mathcal{C}}, \mathbb{G}_{m,\mathcal{C}}) 
\end{equation}
where the log structures are forgotten in the second one (cf. \cite[Corollary 4.2]{Gill2} and \cite[Theorem 2]{Wat} respectively).

\begin{thm} \label{thm prolong log}
For a given model $\mathcal{G}$ of $G$, the map 
\begin{align*}
\Hom(\mathcal{G}^D, \mathcal{J}) & \to \{\mathrm{pointed~N\acute{e}ron-log~} \mathcal{G}\mathrm{-torsors} \}\\
i & \mapsto (i\times id_{\mathcal{C}})^{*}(\mathcal{P}^{log})
\end{align*}

is bijective. In particular, if $Y \to C$ is a pointed fppf $G$-torsor and if the associated morphism $h_Y$ extends into a morphism $i_0: \mathcal{G}^D \to \mathcal{J}$, the torsor $Y$ extends uniquely into the pointed Néron-log $\mathcal{G}$-torsor $(i_0\times id_{\mathcal{C}})^{*}\mathcal{P}^{log}$.\\
Analogously, the map 
\begin{align*}
\Hom(\mathcal{G}^D, \mathcal{J}^0) & \to \{\mathrm{pointed~fppf~} \mathcal{G}\mathrm{-torsors} \}\\
i & \mapsto (i\times id_{\mathcal{C}})^{*}(\mathcal{P}^{0})
\end{align*}
is bijective. In particular, if $i_0$ factors through $\mathcal{J}^0$, then $Y$ extends uniquely into the fppf $\mathcal{G}$-pointed torsor $(i\times id_{\mathcal{C}})^{*}\mathcal{P}^{0}$.
\end{thm}

\begin{proof}
Let $\mathcal{P}^{log}$ be the extension from Corollary~\ref{PoinExt}. Using the morphism $(i \times id_{\mathcal{C}}): \mathcal{G}^D \times \mathcal{C} \to \mathcal{J} \times \mathcal{C}$, we pull-back  $\mathcal{P}^{log}$ and get a logarithmic extension $(i \times id_{\mathcal{C}})^{*}(\mathcal{P}^{log})$ of $\mathcal{G}^D_{\mathcal{C}}$ by $\mathbb{G}_{m,\mathcal{C}}$. Finally, using (\ref{eq:2}), we can associate to the extension $(i \times id_{\mathcal{C}})^{*}(\mathcal{P}^{log})$ a $\mathcal{G}$-log torsor over $\mathcal{C}$. \\
Now, let $T$ be a Néron-log $\mathcal{G}$-torsor over $\mathcal{C}$. It corresponds by the isomorphism (\ref{PointedlogGPic}) to a unique morphism $j: \mathcal{G}^D \to \mathcal{J}$. Generically, by Lemma \ref{Mon lemme}, it is isomorphic to the torsor $(j_K \times id_{C})^{*}(\mathcal{P}_K)$. Hence, $T$ and $(j \times id_{\mathcal{C}})^{*}(\mathcal{P}^{log})$ are both $\mathcal{G}$-torsors over $\mathcal{C}$ with the same generic fiber. It follows from Corollary \ref{uni} that they are the same. Hence, $(i \times id_{\mathcal{C}})^{*}(\mathcal{P}^{log})$ is a Néron-log torsor and the constructed map is surjective. The same arguments and the fact that $\mathcal{J}$ is separated  allow to prove that the map is also injective.\\
Therefore, if $h_Y$ extends into a morphism $i_0: \mathcal{G}^D \to \mathcal{J}$, by the compatibility of the isomorphism \eqref{eq:2} and the isomorphism in Lemma~\ref{Mon lemme}, $Y$ extends into the Néron-log $\mathcal{G}$-torsor $(i_0 \times id_{\mathcal{C}})^{*}(\mathcal{P}^{log})$ and the extension is unique by Corollary \ref{uni}.

The second part is proved analogously using the isomorphism (\ref{eq:2.0}) this time.

\end{proof}

Let $E_Y$ be the unique element of $\Ext_{fppf}^1(G_C^D,\mathbb{G}_{m,C})$ that corresponds to the torsor $Y$ according to (\ref{eq:2.0}). If $h_Y$ extends into $i: \mathcal{G}^D \to \mathcal{J}$, we know from Theorem~\ref{thm prolong log} and (\ref{eq:2}) that $E_Y$ extends uniquely into an element $E_Y^{log}$ of $\Ext^1_{klf}(\mathcal{G}^D_{\mathcal{C}},\mathbb{G}_{m,\mathcal{C}})$, which is explicitly given by  
$$
E_Y^{log}:=(i \times id_{\mathcal{C}})^{*} \mathcal{P}^{log}.
$$

Let us consider the following diagram where the exact lines are extracted from the spectral sequence comparing fppf and log flat cohomology \cite[exact sequence 4.1.4]{Gill1}, and where the vertical arrows are induced by the map $i:\mathcal{G}^D \to \mathcal{J}$:

$$
\xymatrix{
   0 \ar[r] & \Ext^1_{fppf}(\mathcal{J}_{\mathcal{C}},\mathbb{G}_{m,\mathcal{C}})  \ar[d]_-{} \ar[r]^-{\delta} & \Ext^1_{klf}(\mathcal{J}_{\mathcal{C}},\mathbb{G}_{m,\mathcal{C}}) \ar[d]_-{} \ar[r]^-{\gamma} & \Hom_{\mathcal{C}}(\mathcal{J}_{\mathcal{C}},R^1\epsilon_{*}\mathbb{G}_m)  \ar[d]^-{h\mapsto h\circ i}  \\
  0 \ar[r] & \Ext^1_{fppf}(\mathcal{G}^D_{\mathcal{C}},\mathbb{G}_{m,\mathcal{C}})  \ar[r]_-{\alpha} & \Ext^1_{klf}(\mathcal{G}^D_{\mathcal{C}},\mathbb{G}_{m,\mathcal{C}}) \ar[r]_-{\beta} & \Hom_{\mathcal{C}}(\mathcal{G}^D_{\mathcal{C}},R^1\epsilon_{*}\mathbb{G}_m).
}
$$

The map $\beta$ associates to $E_Y^{log}$ some morphism $\mathcal{G}^D_{\mathcal{C}} \to R^1\epsilon_{*}\mathbb{G}_m$. The right part of the diagram being commutative, we have:
\begin{equation}\label{crit2}
\gamma (\mathcal{P}^{log}) \circ i = (\beta \circ (i \times id)^{*})(\mathcal{P}^{log}): \mathcal{G}^D_{\mathcal{C}} \to R^1\epsilon_{*}\mathbb{G}_m.
\end{equation}

\begin{cor} \label{THM_v2}
Let $Y \to C$ be a pointed fppf $G$-torsor. Assume that $h_Y$ extends into an $R$-morphism $i: \mathcal{G}^D \to \mathcal{J}$, where $\mathcal{G}$ is a model of $G$, so that we have a Néron-log extension of the torsor.  Let $\gamma  (\mathcal{P}^{log}) \circ i$ be the morphism associated to the torsor $Y$ defined in ($\ref{crit2}$). Then this extended log torsor comes from an fppf one if and only if $\gamma  (\mathcal{P}^{log}) \circ i=0$. In other words, the obstruction to the extension into an fppf torsor is equal to $\gamma  (\mathcal{P}^{log}) \circ i$. \\

\end{cor}

\begin{prop}
If $\Phi:=\mathcal{J}_k/\mathcal{J}_k^0$ and if $D_1,\dots, D_r$ are the irreducible components of the special fiber $\mathcal{C}_k$ so that $
\mathcal{C}_k=\sum_{i=1}^r n_iD_i$ ,
where the $n_i$ are positive integers, then we can identify $\gamma(\mathcal{P}^{log})$ to a morphism 
$$ \gamma_{\Phi}(\mathcal{P}^{log}):  \Phi  \to  \bigoplus_{i=1}^r (\Q/\Z).D_i $$
which is exactly the obstruction for $\mathcal{P}_K$ to extend into an fppf extension of $\mathcal{J}_{\mathcal{C}}$ by $\mathbb{G}_{m,\mathcal{C}}$ (cf. \cite[Exposé VIII ; remark 7.2]{Groth1}).\\

\end{prop}

\begin{proof}
 $\gamma(\mathcal{P}^{\rm log})$ is a morphism of sheaves for the fppf topology on $\mathcal{C}$. It follows from \cite[Theorem 2.2.2]{Gill1} that the restriction of the sheaf $R^1\epsilon_{*}\mathbb{G}_m$ to the smooth site of $\mathcal{C}$ can be described as follows:
$$
(R^1\epsilon_{*}\mathbb{G}_m)|_{\mathrm{smooth}/\mathcal{C}} \simeq \bigoplus_{i=1}^r (\Q/\Z).D_i
$$
where $(\Q/\Z).D_i$ denotes the skyscraper sheaf supported by $D_i\subset {\mathcal{C}}$, with value $\Q/\Z$. Because $\mathcal{J}_{\mathcal{C}}$ is a smooth group scheme over $\mathcal{C}$, we have
$$
\Hom_{\mathcal{C}}(\mathcal{J}_{\mathcal{C}}, R^1\epsilon_{*}\mathbb{G}_m) = \Hom_{\mathcal{C}}(\mathcal{J}_{\mathcal{C}}, \bigoplus_{i=1}^r (\Q/\Z).D_i) = \bigoplus_{i=1}^r \Hom_{D_i}(\mathcal{J}_{D_i}, \Q/\Z)
$$
in which the second equality follows from the adjunction formula.  \\

\noindent We now observe that, $D_i$ being an irreducible vertical component of $\mathcal{C}$, we have 
$$
\Hom_{D_i}(\mathcal{J}_{D_i}, \Q/\Z) = \Hom_k(\mathcal{J}_k, \Q/\Z)= \Hom_k(\Phi, \Q/\Z) 
$$
where the last equality follows from the fact that $\Hom(\mathcal{J}_k^0,\Q/\Z)=0$.\\

\noindent Going through these isomorphisms, $\gamma(\mathcal{P}^{\rm log}): \mathcal{J}_{\mathcal{C}} \to R^1\epsilon_{*}\mathbb{G}_m$ can be identified with a morphism
$$
\gamma_\Phi(\mathcal{P}^{\rm log}): \Phi \to \bigoplus_{i=1}^r (\Q/\Z).D_i.
$$
For the fact that this is the obstruction for the Poincaré extension to extend into an fppf extension, it is proved in \cite[Theorem 4.1.5(i)]{Gill1}. \\

\end{proof}

\section{A finiteness criterion for $\mathcal{J}[p]$}
\label{section3}
Let $C$ be a smooth and geometrically connected $K$-curve with Jacobian $J$. We assume here that $k$ is algebraically closed. This section is devoted to the generalization of a result of Chiodo \cite[Propositions 7.4.1 and 7.5.1]{Chiodo} which provides a necessary and sufficient condition for $J[r]$, the $r$-torsion subgroup of $J$, to admit a finite \'etale model over $R$, when $r$ is prime to $p$. Such a model is the N\'eron model of $J[r]$, hence is equal to $\mathcal{J}[r]$.

Here, we will treat the case where $p$ divides $r$. In this case, a finite model of $J[r]$ is not necessarily \'etale, hence it is not necessarily the N\'eron model of $J[r]$. But one may still ask for a condition for $\mathcal{J}[r]$ to be finite and flat over $R$. We will see that the result of Chiodo generalizes to this case.

Once this is done, given a finite $K$-group scheme $G$ killed by $r$, we combine this finiteness criterion for $\mathcal{J}[r]$ with Theorem~\ref{thm prolong log} in order to prove  the existence of models of $G$ and extension of $G$-torsors over a regular model of $C$.\\

We start this section with some preliminaries about semistability. For the terminology used about graph theory, we refer to  \cite[\S 2.1 and \S 2.2]{Chiodo}.

A proper curve of genus $\geq 2$ over an algebraically closed field $k$ is \emph{semistable} if it is reduced, connected, has only ordinary double points (i.e. its singularities are nodal), and its irreducible components that are isomorphic to $\mathbb{P}^1_k$ meet the other components in at least two points \cite[Definition 1.1;1.2]{DSCH}.  \\

A proper flat curve $\mathcal{C} \to \Spec(R)$ is \emph{semistable} if its geometric fibers are semistable curves. In particular, given a smooth curve $C$ over $K$, a semistable curve $\mathcal{C} \to \Spec(R)$ with a specified isomorphism $\mathcal{C} \times_{\Spec(R)} \Spec(K) \simeq C$ is called a semistable model of $C$ over $R$ and, by abuse of terminology, one says that $C$ is a semistable curve over $K$.

An abelian variety $A$ over $K$ has semistable reduction over $R$ if the connected component of the special fiber of its N\'eron model $\mathcal{A}_k$ has unipotent rank $0$, which means that $\mathcal{A}_k^0$ is an extension of an abelian variety by a torus. 

This terminology is justified by the following fact: given a smooth, geometrically connected curve $C$ over $K$ of genus $g \geq 2$, the following conditions are equivalent:
\begin{itemize}
\item $C$ has semistable reduction over $R$;
\item the minimal regular model of $C$ over $R$ is semistable;
\item the Jacobian variety of $C$ has semistable reduction over $R$.
\end{itemize}

See \cite[Expos\'e~1, Proposition 2.2]{DSCH} for the equivalence of the first two conditions, and \cite[Expos\'e~1, Proposition 5.7]{DSCH} for the equivalence with the last one.

From now on, let $C$ denote a smooth geometrically connected $K$-curve of genus $g \geq 2$, which has semistable reduction over $R$. Let $\mathcal{C}$ be the minimal regular model of $C$ over $R$. We denote by $\Gamma$ the dual graph of the special fiber $\mathcal{C}_k$: it is the graph whose set of vertices is the set of irreducible components of $\mathcal{C}_k$, and whose set of edges is the set of nodes of $\mathcal{C}_k$; we denote by $b_1(\Gamma)$ the first Betti number of $\Gamma$.

As previously, $J$ denotes the Jacobian of $C$ and $\mathcal{J}$ denotes its N\'eron model over $R$. We denote by $\Phi$ the group of components of the special fiber $\mathcal{J}_k$ of $\mathcal{J}$.

\begin{defi}
In a graph, a \emph{circuit} is a path that begins and ends at the same vertex. 
A circuit that does not repeat vertices is called a \emph{cycle}. 

\noindent If $\Omega$ is a graph, we let $\Cyc(\Omega)$ denote the set of cycles in $\Omega$, and define
$$
c_2(\Omega):=\gcd\{\text{common number of edges of $C$ and $C'$}~|~ C,C'\in \Cyc(\Omega)\}.
$$
\noindent If $\Cyc(\Omega)= \varnothing $, we put $c_2(\Omega)=0$.
\end{defi}

\begin{exam}
If $\Omega$ is a polygon with $d$ edges, then $c_2(\Omega)=d$. If $\Omega$ consists of two polygons with $d$ edges each, sharing $d'$ edges, then $c_2(\Omega)=\gcd(d,d')$. Other examples are given in \cite[\S{}5.6]{Chiodo}.
\end{exam}

With the hypotheses as stated above, the following statement was proved by Chiodo \cite[Proposition 7.5.1]{Chiodo}, using a slightly different terminology:

\begin{thm}(Chiodo's criterion)\label{Chiodo's criterion}.
\label{prop:CC}
We have
$$
\Phi[r] \simeq (\Z/r\Z)^{b_1(\Gamma)}
$$
if and only if $r$ divides $c_2(\Gamma)$.
\end{thm}

The main step in generalizing Chiodo's result is the following:

\begin{thm}
\label{prop:phib1}
Let $r>1$ be an integer. The group scheme $\mathcal{J}[r]$ is finite and flat if and only if the curve $C$ has semistable reduction and the dual graph $\Gamma$ of its special fiber satisfies the condition
$$
\Phi[r] \simeq (\Z/r\Z)^{b_1(\Gamma)}.
$$ 
\end{thm}

\begin{proof}
The case where $r$ is prime to $p$ has been proved by Chiodo. Therefore, we may (and do) assume that $p$ divides $r$.

Firstly, we observe that if $C$ has not semistable reduction, then its Jacobian has not  semistable reduction either. Hence the unipotent rank of $\mathcal{J}_k$ is not $0$, which implies that $\mathcal{J}_{k}$ contains the additive group $\mathbb{G}_a$ as a subgroup. But $\mathbb{G}_{a,k}[p] = \mathbb{G}_{a,k}$ since $p$ is equal to the characteristic of $k$, hence $\mathcal{J}_{k}[r]$ is not a finite group scheme since it contains a one-dimensional subgroup. Therefore, the fact that $C$ has a semistable reduction is a necessary condition for $\mathcal{J}[r]$ to be finite.

We now assume that the minimal regular model $\mathcal{C}$ of $C$ is semistable.
Then, since $J$ has semistable reduction, $\mathcal{J}[r]$ is quasi-finite and flat (see \cite[\S 7.3, Lemma~2]{BLR}). But  $\mathcal{J}[r]$ is finite if and only if its generic and special fibers have the same rank, i.e. $\mathcal{J}_k[r]$ has rank $r^{2g}$ as a finite $k$-scheme. On the other hand, by the snake Lemma, we have a long exact sequence:
$$
0 \to \mathcal{J}_k^0[r] \to \mathcal{J}_k[r] \to \Phi[r] \to \mathcal{J}_k^0/r \to \cdots
$$

Since $J$ is semistable, the multiplication map $\mathcal{J}_k^0 \xrightarrow{~r~} \mathcal{J}_k^0$ is surjective \cite[\S 7.3, Lemma~1]{BLR}, hence $\mathcal{J}_k^0/r=0$. Therefore, we have
\begin{equation}
\label{eq:rank}
\rank \mathcal{J}_k[r] = \rank \mathcal{J}_k^0[r] \times \rank \Phi[r].
\end{equation}

The semistability of $C$ implies that $\Pic^0_{\mathcal{C}} \simeq \mathcal{J}^0$  \cite[\S 9.5, Theorem~4]{BLR}, and that $\Pic^0_{\mathcal{C}_{k}}$ is extension of an abelian variety of rank $a$ by a torus of rank $b_1({\Gamma})$ \cite[\S9.2, Example~8]{BLR}. We have
$$
a+b_1({\Gamma}) = \dim_k H^1(\mathcal{C}_k, \mathcal{O}_{\mathcal{C}_k})= \dim_K H^1(C,\mathcal{O}_C)=g
$$
where the first equality follows from \cite[\S 7.5, Definition 5.21]{Liu}, and the second one, given the flatness and projectivity of $\mathcal{C}$ over $R$, follows from \cite[\S 8.3, Corollary~8.3.6]{Liu}. 

Finally, it is well-known that the $r$-torsion subgroup of an abelian variety (resp. a torus) of dimension $d$ is a finite group scheme of rank $r^{2d}$ (resp. $r^d$). Putting everything together, we deduce that
$$
\rank\Pic^0_{\mathcal{C}_{k}}[r] = r^{2a}\times r^{b_1({\Gamma})}=r^{2g-b_{1}(\Gamma)}
$$
as a finite $k$-group scheme. Hence, it follows from \eqref{eq:rank} that $\mathcal{J}_k[r]$ has rank $r^{2g}$ if and only if $\Phi[r]$ has rank $r^{b_1(\Gamma)}$ as a finite $k$-group scheme. Since $\Phi[r]$ is \'etale, this means that $\Phi[r]$ has exactly $r^{b_1(\Gamma)}$ points over $k$. By \cite[Lemma 7.3.4]{Chiodo}, this is equivalent to $\Phi_{k}[r] \simeq (\Z/r\Z)^{b_1(\Gamma)}$.

\end{proof}

\begin{cor}
\label{cor Chiodo}
Let $r>1$ be an integer. The group scheme $\mathcal{J}[r]$ is finite flat if and only if the curve $C$ has semistable reduction and the dual graph $\Gamma$ of the special fiber of $\mathcal{C}$ satisfies Chiodo's criterion.
\end{cor}

\begin{cor} \label{coro}
Let $C$ be a smooth projective  geometrically connected curve over $K$ of genus $g \geq 2$, with a $K$-rational point. Let $G$ be a commutative finite  flat  $K$-group scheme killed by $r$, and let $Y \rightarrow C$ be a pointed  fppf $G$-torsor such that $Y$ is geometrically connected. If $\mathcal{C}$ has semistable reduction and if Chiodo's criterion is satisfied, then $G$ has an $R$-model $\mathcal{G}$ and $Y \rightarrow C$ extends uniquely into a pointed logarithmic $\mathcal{G}$-torsor over any regular model of $C$.
\end{cor}

\begin{proof}
Since $G$ is killed by $r$, so is $G^D$, hence the associated morphism $h_Y$ factors through $J[r]$. In addition, given that  $Y$ is geometrically connected, $h_Y$ is a closed immersion (cf. \cite[VI, \S 11, Proposition 10]{Serre}); hence, $G^D$ can be seen as a subscheme of $J[r]$ so that we can consider $\overline{G^D}$, the schematic closure of $G^D$ inside $\mathcal{J}[r]$. According to \cite[Proposition 2.8.5]{Groth4}, it is the unique closed subscheme of $\mathcal{J}[r]$ that is flat over $\Spec(R)$ and whose generic fiber is $G^D$. Furthermore, since taking the schematic closure commutes with fibered products over $\Spec(R)$ (\cite[Corollary 2.8.6]{Groth4}), the multiplication over $G^D$ extends naturally into $\overline{G^D}$. So $\overline{G^D}$ is a closed and flat subgroup scheme of $\mathcal{J}[r]$.\\
Denote by $\mathcal{C}$ the minimal regular model of $C$ over $R$. Now, by assumption and Corollary \ref{cor Chiodo}, the $R$-group scheme $\mathcal{J}[r]$ is finite, hence so is $\overline{G^D}$. By biduality, the Cartier dual $\mathcal{G}$ of $\overline{G^D}$ comes with a closed immersion $\mathcal{G}^D \to \mathcal{J}[r]$ extending $G^D \to J[r]$. Whence the corollary by Theorem~\ref{thm prolong log}.

\end{proof}

\section{An example of extension of torsors over a hyperelliptic curve}\label{section4}

In this section, we will study an example of extension of torsors on a given curve to illustrate the results above. For some rational number $c \notin \{  \pm 1,0 \}$ and a prime $p \neq 2$, we consider the hyperelliptic curve $C$ defined over $\mathbb{Q}$ by:
$$ y^2=f(x)=x^{2p}-(1+c^2)x^p+c^2 $$

\noindent (the example is taken from \cite{GL}). Note that $Q=(1,0)$ is a rational point over this curve, so that we can consider pointed torsors over $C$ relative to $Q$.  \\

\subsection{A hyperelliptic curve whose Jacobian contains a subgroup isomorphic to $(\Z/p\Z)^2$} \label{jacobian contains a subgroup}
What follows comes from \cite[Lemma 3.3]{GL}. Given that in the projective space, the curve $C$ has a singularity at one of the two points at infinity, one checks that we have a smooth projective model of the curve over $\Q$. It is covered by the following two affine charts: 

\begin{itemize}
\item  $y^2=f(x)=x^{2p}-(1+c^2)x^p+c^2$ \hop
\item $t^2=s^{2p}f(\frac{1}{s})$
\end{itemize} \hop

\noindent that glue together using $(x,y)=(\frac{1}{s},\frac{t}{s^p})$. We also denote indifferently this model by $C$.

\noindent Using the Jacobian criterion, we can check that the curve has bad reduction over the primes $2$, $p$ and the primes that divide $c(c^2-1)$.\\
The $p$-torsion in the Jacobian of $C$ comes from the relations:
 \begin{equation}\label{divis}
   (y-x^p-c)(y+x^p+c)=-(1+c)^2x^p 
 \end{equation}
$$(y-x^p+c)(y+x^p-c)=-(1-c)^2x^p.  $$
We compute that 
$$ \mathrm{div}(y-x^p-c)=p(0,c)-p\infty$$
$$  \mathrm{div}(y-x^p+c)=p(0,-c)-p\infty$$
where $\infty$ denotes one of the two points at infinity on $C$. Hence, the two divisors 
\begin{center}
$(0,c)-\infty$ and $(0,-c) -\infty$
\end{center}

\noindent define two classes of order $p$ in the Jacobian of $C$ that are independent classes of divisors by Riemann-Roch theorem. So the Jacobian of the curve $C$ contains a subgroup isomorphic to $(\Z/p\Z)^2$. Therefore, we have a pointed $\mu_p^2$-torsor over $C$.

\subsection{Construction of a regular model of $C$ and study of extension of pointed torsors over it}
In this section, we restrict to the case where $p=c=3$.\\
Let us consider the arithmetic surface $\mathcal{C}/\Z$ with generic fiber $C$, given in affine charts by: 
\begin{itemize}
\item  $y^2=f(x)=x^{6}-10x^3+3^2$ \hop
\item $t^2=s^{6}f(\frac{1}{s}).$
\end{itemize} \hop

\begin{itemize}
\item \textbf{Construction of a regular model $\mathcal{C}_3$ over $\Z_3$}.
Using the Jacobian criterion, we check that in each affine chart, there are two singular points in the fiber over the prime $3$, and in fact, they belong to the intersection of the two charts. So, it suffices to desingularize in only one of the two charts; for example, the first one. \\
In the first chart, the Jacobian matrix at a point $(x,y)$ of  $\mathcal{C}$ is given by:
\begin{center}
$J(x,y)=(-6x^{5}+30x^{2} ~~~~2y)\equiv (0~~~~2y)\mod~3$.
\end{center}
Therefore, the eventual singularities are the solutions of the system:

$$
\left\{
	\begin{aligned}
	& 2y \equiv 0 \mod~3\\	
	& y^2=(x^3-1)x^3 \equiv (x-1)x \mod~3.
	\end{aligned}
\right.
$$

This gives two points that are eventually singular: $\mathfrak{M}_1:=<x-1,y,3>$ and\\ $\mathfrak{M}_2:=<x,y,3>$. While the Krull dimension of the local ring $(\Z[X,Y]/(Y^2-f(X)))_{\mathfrak{M}_1}$ is $2$, we compute that $\mathrm{dim}_{\mathbb{F}_3} \mathfrak{M}_1/ \mathfrak{M}_1^2=3$. Thus, $\mathfrak{M}_1$ is not a regular point and the same holds for $\mathfrak{M}_2$. Next, we will blow up separately $\mathcal{C}$ in each of the closed points $\mathfrak{M}_1$ and $\mathfrak{M}_2$.\hop
We view our chart of $\mathcal{C}$ inside $\mathbb{A}^2_{\Z}$. Then, the blow-up of $\mathcal{C}$ at $\mathfrak{M}_1$ is formed by taking the following three schemes and gluing them together:

\textbf{Chart 1: $\mathcal{C}_{3,1}^1$}. Define new variables $y=(x-1)v$ and $3=(x-1)w$. Let us write $f(x)= \sum_{i=0}^{6} a_i(x-1)^i=\sum_{i=2}^{6} a_i(x-1)^i-8.3(x-1)$ since $a_0=0$ and $a_1=-3.8$. After replacing $y$ and $3$ by the new variables in the equation of the first chart, we get that $\mathcal{C}_{3,1}^1$ is given by the system:

$$ \mathcal{C}_{3,1}^1 
\left\{
	\begin{aligned}
	& v^2= \sum_{i=2}^{6} a_{i}(x-1)^{i-2}-8w \\	
	& (x-1)w=3. \\
	\end{aligned} 
\right.
$$

\textbf{Chart 2: $\mathcal{C}_{3,1}^2$}. The second chart is formed using the new variables $x-1=yu$ and $3=yt$. Replacing in the equation of $\mathcal{C}$ gives: 

$$ \mathcal{C}_{3,1}^2 
\left\{
	\begin{aligned}
	& 1=-8tu+\sum_{i=2}^{6}a_{i}y^{i-2}u^{i} \\	
	& yt=3. \\
	\end{aligned}
\right.
$$

\textbf{Chart 3: $\mathcal{C}_{3,1}^3$}. The third chart is formed using the new variables $x-1=3f$ and $y=3h$. We get:

$$ \mathcal{C}_{3,1}^3: 
	 h^2=-8f+\sum_{i=2}^{6} a_{i}3^{i-2}f^{i}. $$

Then, we glue the four charts together (the three charts defined above and the chart given by $t^2=s^{6}f(\frac{1}{s})$) using the change of variables defined above, and this gives an arithmetic surface over $\Z$ with generic fiber isomorphic to the curve $C$. We claim that this surface is no longer singular over $\mathfrak{M}_1$. Indeed, we just have to check the regularity over $\mathfrak{M}_1$ in one of the three affine charts defined above. For instance, in the first chart $\mathcal{C}_{3,1}^1$, the Jacobian matrix at a point $(x,v,w)$ is given by:

\begin{center}
 $J(x,v,w) = \begin{pmatrix}- \sum_{i=3}^{6}a_{i}(i-2)(x-1)^{i-3} & 2v & -8 \\ w & 0 & x-1 \end{pmatrix}.$
\end{center}

Since we look for singularities after the blow-up at $\mathfrak{M}_1$, we focus on points that verify $x=1$, $y=0$ and $3=0$. We get:

\begin{center}
 $J(1,v,w) = \begin{pmatrix}-a_3 & 2v & -8\\ w & 0 & 0 \end{pmatrix}.$
\end{center}
Hence we get the system  (everything $\mod~3$):

$$ \left\{
	\begin{aligned}
	& v^2 \equiv a_2-8w \\	
	& 2vw \equiv 0\\
	& 8w \equiv 0. \\
	\end{aligned}
 \right.$$

Thus, $w \equiv 0$. The first equation then gives $v^2-a_2 \equiv 0$; an easy computation of $a_2$ gives that $a_2 \equiv 0 \mod~3$ and thus $v \equiv 0$. Therefore, $\mathfrak{M}_1':=<x-1,v,w,3>$ is the only possible singularity and we can check that in fact, this point is regular. Indeed, we need to see that $3$ and $w$ live in $\mathfrak{M}_1'^2$. For $3$, it is clear since $3=(x-1)w$. As for $w$:

$$0 \equiv v^2- \sum_{i=2}^{6}a_{i-2}(x-1)^i+8w \equiv a_2+ a_3(x-1)+8w \mod \mathfrak{M}_1'^2.$$

Since $a_2$ and $a_3$ are both multiples of $3$ (easy to check), it follows that $w \in \mathfrak{M}_1'^2$. Therefore, $\mathfrak{M}_1'$ is regular. Finally, we conclude that the surface is no longer singular over $\mathfrak{M}_1$. We base change to $\Spec(\Z_3)$ and we get a model of the curve $C$ over $\Z_3$; let us denote it by $\mathcal{C}_{3,1}$. \hop \hop
We do exactly the same work at the other singular point of $\mathcal{C}$: the blow-up of $\mathcal{C}$ at $\mathfrak{M}_2$ is formed by taking the following three schemes and gluing them together, using the suitable change of variables. \\

\textbf{Chart 1: $\mathcal{C}_{3,2}^1$}. We define new variables $y=x\alpha$ and $3=x\beta$. We get

$$ \mathcal{C}_{3,2}^1 
\left\{
	\begin{aligned}
	& \alpha^2= x^4-10x+\beta^2 \\	
	& x\beta=3. \\
	\end{aligned}
\right.$$

\textbf{Chart 2: $\mathcal{C}_{3,2}^2$}. We define new variables $x=y\gamma$ and $3=y\delta$:

$$ \mathcal{C}_{3,2}^2 
\left\{
	\begin{aligned}
	& 1=\gamma^6y^4-10\gamma^3y+\delta^2 \\	
	& y\delta=3. \\
	\end{aligned}
\right.$$
	
	\textbf{Chart 3: $\mathcal{C}_{3,2}^3$}. We define new variables $x=3\eta$ and $y=3\mu$:

$$ \mathcal{C}_{3,2}^3: \mu^2= 3^4\eta^6 -10.3\eta^3+1.$$

We glue the four charts together using the change of variables and this gives a model of $C$ over $\Z$. Once again, we can check that this model is no longer singular over $\mathfrak{M}_2$. For instance, in the first affine chart $\mathcal{C}_{3,2}^1$, the Jacobian matrix at a point $(x, \alpha,\beta)$ is given by 
\begin{center}
 $J(x,\alpha,\beta) = \begin{pmatrix} -4x^3+10 & 2\alpha & -2\beta \\ \beta & 0 & x \end{pmatrix}.$
\end{center}
When $x=0, y=0$ and $3=0$, the matrix becomes
$$J(0,\alpha,\beta) = \begin{pmatrix} 10 & 2\alpha & -2\beta \\ \beta & 0 & 0 \end{pmatrix}.$$
Hence we get the system  (everything $\mod~3$):
$$ \left\{
	\begin{aligned}
	& \alpha^2 \equiv \beta^2 \\	
	& \beta^2 \equiv 0\\
	& 2\alpha\beta \equiv 0. \\
	\end{aligned}
 \right.$$
Therefore, $\alpha \equiv \beta \equiv 0$, which means that there is only one eventual singular point which is $\mathfrak{M}_2':=<x,y,\alpha,\beta,3>$. But $3=x\beta \in \mathfrak{M}_2'^2$, $y=x\alpha \in \mathfrak{M}_2'^2$ and from $\alpha^2 \equiv x^4-10x+\beta^2$, we deduce that $x \in \mathfrak{M}_2'^2$. Hence $\mathfrak{M}_2'$ is regular.\\
We denote by $\mathcal{C}_{3,2}$ the model over $\Z_3$ obtained by doing the base change to $\Z_3$ of the gluing of the four charts above. Finally, we denote by $\mathcal{C}_{3}$ the gluing of $\mathcal{C}_{3,1}$ and $\mathcal{C}_{3,2}$. By the work done before, $\mathcal{C}_{3}$ is a \textit{regular} model of $C$ over $\Z_3$.

\item \textbf{Computation of the special fiber of $\mathcal{C}_3$}.\hop 
In this section, we compute the special fiber of the regular model $\mathcal{C}_3$. One can check that after gluing together all the charts that cover $\mathcal{C}_3$, the special fiber of $\mathcal{C}_3$, that we denote by $\tilde{\mathcal{C}}_3$, lies over the charts $\mathcal{C}_{3,1}^1$ and $\mathcal{C}_{3,2}^1$, except for some points at infinity. \hop
There are two irreducible components of the special fiber that lie inside $\mathcal{C}_{3,1}^1$; they  intersect in one point ($x \equiv 1,v \equiv 0,w \equiv 0 \mod~3$):

$$
F_{3,1}^1 \left\{
	\begin{aligned}
	& x \equiv 1 \\	
	& v^2 \equiv w. \\
	\end{aligned}\right. ~~~~~
	\nolinebreak
 F_{3,1}^2 \left\{
	\begin{aligned}
	& w \equiv 0 \\	
	& v^2 \equiv x^4-x.  \\
	\end{aligned}\right.
$$

There are three irreducible components of the special fiber that lie inside $\mathcal{C}_{3,2}^1$ and they intersect in one point ($x \equiv 0, \alpha \equiv 0, \beta \equiv 0 \mod~3$):

$$
F_{3,2}^1 \left\{
	\begin{aligned}
	& x \equiv 0 \\	
	& \alpha \equiv \beta. \\
	\end{aligned}\right. ~~~~~
	\nolinebreak
	F_{3,2}^2 \left\{
	\begin{aligned}
	& x \equiv 0 \\	
	& \alpha \equiv -\beta. \\
	\end{aligned}\right. ~~~~~
	\nolinebreak
 F_{3,2}^3 \left\{
	\begin{aligned}
	& \beta  \equiv 0 \\	
	& \alpha^2 =x^4-x.  \\
	\end{aligned}\right.
$$

Notice that $F_{3,1}^1$ (resp. each of $F_{3,2}^1$ and $F_{3,2}^2$) is contained in the exceptional divisor produced after the first (resp. second) blow-up, while $F_{3,1}^2$ (resp. $F_{3,2}^3$) meets the exceptional divisor in only one point. Thus, after gluing together $\mathcal{C}_{3,1}$ and $\mathcal{C}_{3,2}$, the components $F_{3,1}^2$ and $F_{3,2}^3$ become identified together since blowing up did not make any change in points where we did not blow up. Therefore, $\tilde{\mathcal{C}}_{3}$ is composed of four components (see figure \ref{fig:fsp} below).

\begin{figure}[h!]

\centering
\includegraphics[scale=0.5]{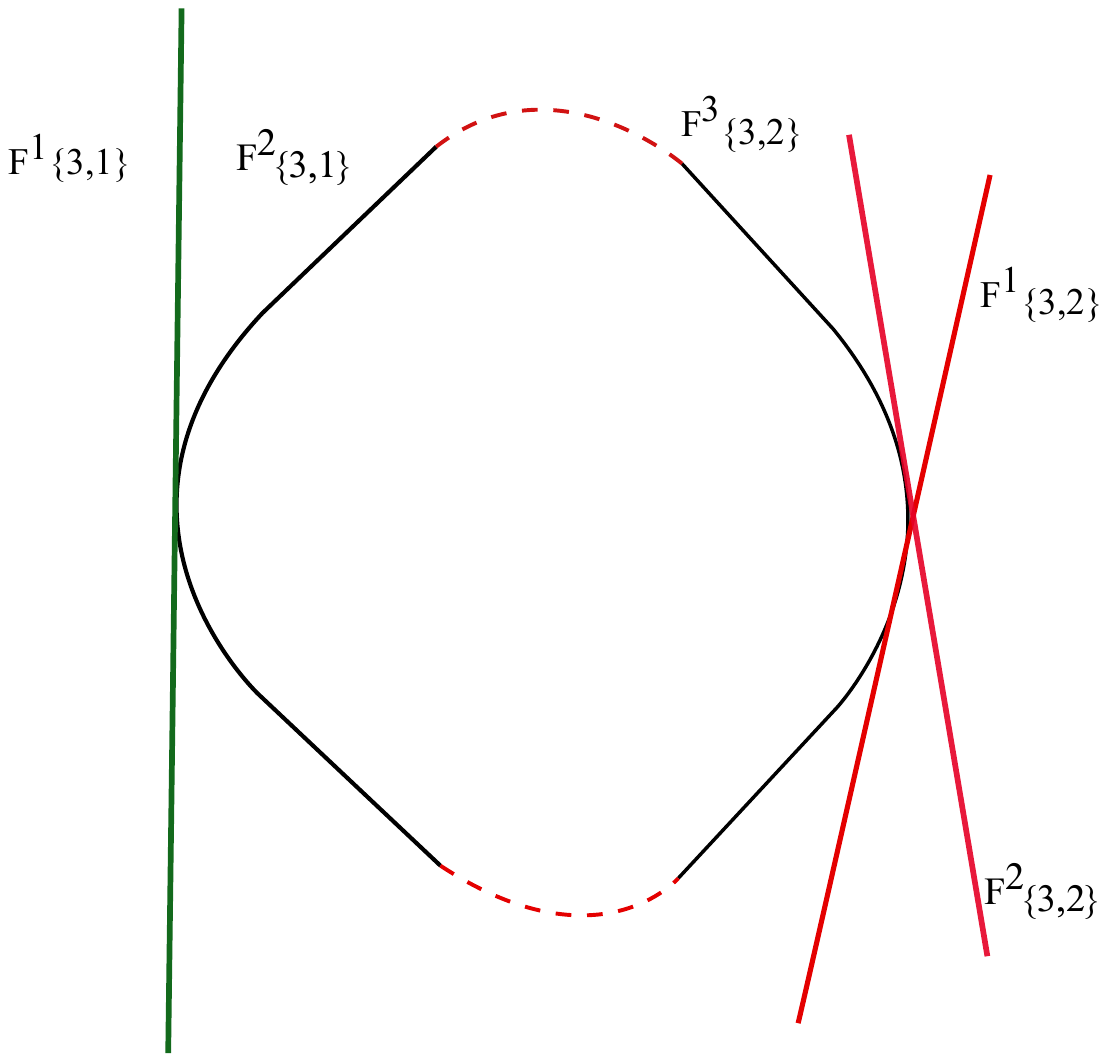} \caption{The special fiber $\tilde{\mathcal{C}}_3$.}
\label{fig:fsp}
\end{figure}

Using \cite[\S 9.6, Theorem 1]{BLR}, one can compute the group of components of the special fiber $\tilde{\mathcal{C}}_3$ and its order is given by the $gcd$ of the minors of the following intersection matix:

\begin{center}
 $\begin{pmatrix}F_{3,1}^1F_{3,1}^1 & F_{3,1}^1F_{3,1}^2 & F_{3,1}^1F_{3,2}^1 &F_{3,1}^1F_{3,2}^2 \\ F_{3,1}^2F_{3,1}^1 & F_{3,1}^2F_{3,1}^2 & F_{3,1}^2F_{3,2}^1 &F_{3,1}^2F_{3,2}^2
 \\  F_{3,2}^1F_{3,1}^1 & F_{3,2}^1F_{3,1}^2 & F_{3,2}^1F_{3,2}^1 &F_{3,2}^1F_{3,2}^2 \\
  F_{3,2}^2F_{3,1}^1 & F_{3,2}^2F_{3,1}^2 & F_{3,2}^2F_{3,2}^1 &F_{3,2}^2F_{3,2}^2
 \end{pmatrix}=\begin{pmatrix}-2 & 2 & 0 & 0 \\ 2 & -4 & 1 & 1
 \\ 0 & 1 & -2 & 1 \\
  0 & 1 & 1 & -2
 \end{pmatrix},$
\end{center}

so it is $6$.

\item \textbf{Study of the extension of the pointed $\mu_3^2$-torsor}.

As shown in subsection \ref{jacobian contains a subgroup}, the Jacobian of the curve $C$ contains a subgroup isomorphic to $(\Z/3 \Z)^2$; thus, we have a pointed $\mu_3^2$-torsor over $C$. Let $\mathcal{J} \to \Z$ be the N\'eron model of $J$ over $\Z$; then, the inclusion $(\Z/3\Z)^2 \subset J$ extends to a morphism $\mathcal{H}:= (\Z/3\Z)^2 \to \mathcal{J}$. Let us now consider $\mathcal{H}_3:= \mathcal{H} \times_{\Z} \Z_3 \to \mathcal{J}_3$ (where $\mathcal{J}_3:= \mathcal{J} \times_{\Z} \Z_3$). By endowing $\mathcal{C}_3$ with the logarithmic structure induced by its special fiber, and by the work done in section \ref{section1}, we deduce that the $\mu_3^2$-pointed torsor over $C$ considered before extends uniquely to a logarithmic $\mu_3^2$-torsor over $\mathcal{C}_3$. The question is, then, whether or not it extends to a pointed fppf-torsor over $\mathcal{C}_3$. According to Proposition \ref{G^D->J^0}, this is equivalent to verify if the image of $\mathcal{H}_3$ lands into $\mathcal{J}^0_3$. This amounts to ask if the two divisors 
\begin{center}
$(0,3) - \infty$ and $(0,-3) - \infty$
\end{center}
extend into sections of $\mathcal{J}^0_3$ over $\Z_3$.\\

On the other hand, we recall that the data of a torsor is equivalent to that of a $ \mathbb{G}_m$-torsor which verifies the theorem of the square (cf. Lemma \ref{lem:SGA7}). Therefore, since the $\mu_3$-torsor over $C$ extends uniquely into a log torsor over $\mathcal{C}_3$, this implies that the $\mathbb{G}_m$-torsors $(0,3) - \infty$ and $(0,-3) - \infty$ extend into $\mathbb{G}_m$-logarithmic torsors (i.e. divisors with rational coefficients).\\

 Another way to say the same thing is: since $\mathcal{C}_3$ is flat, the curves $C$ and $\mathcal{C}_3$ have the same field of functions, hence we can consider two natural divisors over $\mathcal{C}_3$
\begin{center}
$D_+:=\frac{1}{3} \mathrm{div}(y-x^3-3)= \overline{(0,3)}-\overline{\infty}+\frac{1}{3}V_3$ \\
$D_-=\frac{1}{3} \mathrm{div}(y-x^3+3)= \overline{(0,-3)}-\overline{\infty}+\frac{1}{3}V_3'$

\end{center}
where $\overline{(0,3)}$, $\overline{(0,-3)}$ and $\overline{\infty}$ are horizontal divisors corresponding to the sections of $\mathcal{C}_3$ that extend the points of $C$, and $V_3$ and $V_3'$ are two vertical divisors over $\mathcal{C}_3$.\\
 The extended torsor is fppf if and only if the coefficients of $D_+$ and $D_-$ are integral. Indeed, $D_+$ and $D_-$ induce a subgroup of $3$-torsion inside $\Pic^{log}_{\mathcal{C}_3/R}$, hence a morphism $h: (\Z/3\Z)^2 \to \Pic^{log}_{\mathcal{C}_3/R}$ which extends by construction the generic morphism $h_K: (\Z/3\Z)^2 \hookrightarrow J$ considered above. Therefore, this gives by (\ref{PointedlogGPic}) a logarithmic $\mu_3^2$-torsor over $\mathcal{C}_3$ which extends the generic $\mu_3^2$-torsor over $C$.\\
  On the other hand, by the Néron mapping property,  $h_K$ also extends into $(\Z/3\Z)^2 \to \mathcal{J}$, so by unicity of extension of torsors (cf. Corollary \ref{uni}), the morphism $h$ factors through $\mathcal{J}$ and the extended log torsor is a Néron-log torsor. Therefore, according to \S~\ref{obstr}, this extended torsor is fppf if and only if the coefficients appearing in the corresponding divisors are integral. \\
 
We need to compute $\mathrm{div}(y-x^3-3)$ and $\mathrm{div}(y-x^3+3)$ over the special fiber $\tilde{\mathcal{C}}_3$. For instance, over the special fiber, the first function vanishes along the exceptional divisor which is obtained after the second blow-up. Let us compute its multiplicity. For example, in the chart $\mathcal{C}_{3,2}^1$ where $y=x\alpha$ and $3=x\beta$, and where the curve is given by 
$$  
\left\{
	\begin{aligned}
	& \alpha^2= x^4-10x+\beta^2 \\	
	& x\beta=3, \\
	\end{aligned}
\right.$$
we have 
$$\mathrm{div}(y-x^3-3)=\mathrm{div}(x)+\mathrm{div}(\alpha-x^2-\beta). $$
We compute that $\mathrm{div}(x)=\mathrm{div}\big(\frac{(\alpha-\beta)(\alpha+\beta)}{x^3-1}\big)=F_{3,2}^1+F_{3,2}^2$. Now, using the relation (\ref{divis}) and $\mathrm{div}(x)$, we compute that $\mathrm{div}(\alpha-x^2-\beta)=F_{3,2}^1$. Hence, 
\begin{center}
$D_+= \frac{1}{3} \mathrm{div}(y-x^3-3)= \overline{(0,3)}-\overline{\infty}+\frac{1}{3}(2F_{3,2}^1+F_{3,2}^2).$ \\
\end{center}
In the same way, we get 
\begin{center}
$D_-=\frac{1}{3} \mathrm{div}(y-x^3+3)= \overline{(0,-3)}-\overline{\infty}+\frac{1}{3}(2F_{3,2}^2+F_{3,2}^1).$ \\
\end{center}

\textbf{Conclusion}. 
Since the coefficients of $D_+$ and $D_-$ are not integral, we conclude that the pointed $\mu_3^2$-torsor does not extend into a pointed fppf-torsor over $\mathcal{C}_3$. 
Moreover, since there is no extension into an fppf torsor, this means by Remark \ref{premier} that the order of the group of components is not prime to $3$, which is indeed the case here since we computed that this order is $6$.
\end{itemize}

\bibliographystyle{plain}
\bibliography{Mehidi}

\begin{thebibliography}{10}

\bibitem{Groth1}
{\em Groupes de monodromie en g\'{e}om\'{e}trie alg\'{e}brique. {I}}.
\newblock Lecture Notes in Mathematics, Vol. 288. Springer-Verlag, Berlin-New
  York, 1972.
\newblock S\'{e}minaire de G\'{e}om\'{e}trie Alg\'{e}brique du Bois-Marie
  1967--1969 (SGA 7 I), Dirig\'{e} par A. Grothendieck. Avec la collaboration
  de M. Raynaud et D. S. Rim.

\bibitem{DSCH}
{\em S\'{e}minaire sur les {P}inceaux de {C}ourbes de {G}enre au {M}oins
  {D}eux}.
\newblock Soci\'{e}t\'{e} Math\'{e}matique de France, Paris, 1981.
\newblock Ast\'{e}risque No. 86 (1981) (1981).

\bibitem{Groth2}
{\em Rev\^{e}tements \'{e}tales et groupe fondamental ({SGA} 1)}, volume~3 of
  {\em Documents Math\'{e}matiques (Paris) [Mathematical Documents (Paris)]}.
\newblock Soci\'{e}t\'{e} Math\'{e}matique de France, Paris, 2003.
\newblock S\'{e}minaire de g\'{e}om\'{e}trie alg\'{e}brique du Bois Marie
  1960--61. [Algebraic Geometry Seminar of Bois Marie 1960-61], Directed by A.
  Grothendieck, With two papers by M. Raynaud, Updated and annotated reprint of
  the 1971 original [Lecture Notes in Math., 224, Springer, Berlin; MR0354651
  (50 \#7129)].

\bibitem{Ant}
Marco Antei.
\newblock On the abelian fundamental group scheme of a family of varieties.
\newblock {\em Israel J. Math.}, 186:427--446, 2011.

\bibitem{Antei}
Marco Antei and Michel Emsalem.
\newblock Models of torsors and the fundamental group scheme.
\newblock {\em Nagoya Math. J.}, 230:18--34, 2018.

\bibitem{Alb}
Alberto Bellardini.
\newblock On the log-picard functor for aligned degenerations of curves.

\bibitem{BLR}
Siegfried Bosch, Werner L\"{u}tkebohmert, and Michel Raynaud.
\newblock {\em N\'{e}ron models}, volume~21 of {\em Ergebnisse der Mathematik
  und ihrer Grenzgebiete (3) [Results in Mathematics and Related Areas (3)]}.
\newblock Springer-Verlag, Berlin, 1990.

\bibitem{Chiodo}
Alessandro Chiodo.
\newblock Quantitative {N}\'{e}ron theory for torsion bundles.
\newblock {\em Manuscripta Math.}, 129(3):337--368, 2009.

\bibitem{Gill1}
Jean Gillibert.
\newblock Prolongement de biextensions et accouplements en cohomologie log
  plate.
\newblock {\em Int. Math. Res. Not. IMRN}, (18):3417--3444, 2009.

\bibitem{Gill2}
Jean Gillibert.
\newblock Cohomologie log plate, actions mod\'{e}r\'{e}es et structures
  galoisiennes.
\newblock {\em J. Reine Angew. Math.}, 666:1--33, 2012.

\bibitem{GL}
Jean Gillibert and Aaron Levin.
\newblock Pulling back torsion line bundles to ideal classes.
\newblock {\em Math. Res. Lett.}, 19(6):1171--1184, 2012.

\bibitem{Groth4}
A.~Grothendieck.
\newblock \'{E}l\'{e}ments de g\'{e}om\'{e}trie alg\'{e}brique. {IV}. \'{E}tude
  locale des sch\'{e}mas et des morphismes de sch\'{e}mas {IV}.
\newblock {\em Inst. Hautes \'{E}tudes Sci. Publ. Math.}, (32):361, 1967.

\bibitem{illusie}
Luc Illusie.
\newblock An overview of the work of {K}. {F}ujiwara, {K}. {K}ato, and {C}.
  {N}akayama on logarithmic \'{e}tale cohomology.
\newblock Number 279, pages 271--322. 2002.
\newblock Cohomologies $p$-adiques et applications arithm\'{e}tiques, II.

\bibitem{Kajiwara}
Takeshi Kajiwara.
\newblock Logarithmic compactifications of the generalized {J}acobian variety.
\newblock {\em J. Fac. Sci. Univ. Tokyo Sect. IA Math.}, 40(2):473--502, 1993.

\bibitem{kato}
Kazuya Kato.
\newblock Logarithmic structures of {F}ontaine-{I}llusie.
\newblock In {\em Algebraic analysis, geometry, and number theory ({B}altimore,
  {MD}, 1988)}, pages 191--224. Johns Hopkins Univ. Press, Baltimore, MD, 1989.

\bibitem{Liu}
Qing Liu.
\newblock {\em Algebraic geometry and arithmetic curves}, volume~6 of {\em
  Oxford Graduate Texts in Mathematics}.
\newblock Oxford University Press, Oxford, 2002.
\newblock Translated from the French by Reinie Ern\'{e}, Oxford Science
  Publications.

\bibitem{Olsson}
Martin~C. Olsson.
\newblock Semistable degenerations and period spaces for polarized {$K3$}
  surfaces.
\newblock {\em Duke Math. J.}, 125(1):121--203, 2004.

\bibitem{Raynaud}
M.~Raynaud.
\newblock Sp\'{e}cialisation du foncteur de {P}icard.
\newblock {\em Inst. Hautes \'{E}tudes Sci. Publ. Math.}, (38):27--76, 1970.

\bibitem{Serre}
Jean-Pierre Serre.
\newblock {\em Algebraic groups and class fields}, volume 117 of {\em Graduate
  Texts in Mathematics}.
\newblock Springer-Verlag, New York, 1988.
\newblock Translated from the French.

\bibitem{Tossici}
Dajano Tossici.
\newblock Effective models and extension of torsors over a discrete valuation
  ring of unequal characteristic.
\newblock {\em Int. Math. Res. Not. IMRN}, pages Art. ID rnn111, 68, 2008.

\bibitem{Wat}
William~C. Waterhouse.
\newblock Principal homogeneous spaces and group scheme extensions.
\newblock {\em Trans. Amer. Math. Soc.}, 153:181--189, 1971.

\end{thebibliography}
\textit{Data sharing not applicable to this article as no datasets were generated or analysed during the
current study.}\\
Sara Mehidi: Institut de Mathématiques de Bordeaux,\\
351, cours de la Libération - F 33 405 TALENCE. \\
Bureau 315, IMB. \\
E-mail address: \texttt{sarah.mehidi@math.u-bordeaux.fr}

\end{document}